\documentclass[]{interact}
\usepackage[font={footnotesize,it}]{caption}
\usepackage{tikz}
\usetikzlibrary{quotes,angles}
\usepackage{bbm}
\usepackage{hyperref}

\def\R{\mathbb R}
\def\C{\mathbb C}

\usepackage{epstopdf}
\usepackage[caption=false]{subfig}

\usepackage[numbers,sort&compress]{natbib}
\bibpunct[, ]{[}{]}{,}{n}{,}{,}
\makeatletter
\def\NAT@def@citea{\def\@citea{\NAT@separator}}
\makeatother

\theoremstyle{plain}
\newtheorem{theorem}{Theorem}[section]
\newtheorem{lemma}[theorem]{Lemma}
\newtheorem{corollary}[theorem]{Corollary}
\newtheorem{proposition}[theorem]{Proposition}

\theoremstyle{definition}
\newtheorem{definition}[theorem]{Definition}

\theoremstyle{remark}

\begin{document}


\title{Location of Ritz values in the numerical range of normal matrices$^\dagger$\thanks{$^\dagger$This article has been accepted for publication in Linear and Multilinear Algebra, published by Taylor \& Francis.}}

\author{\name{Kennett L. Dela Rosa\textsuperscript{a,b}\thanks{K.~L. Dela Rosa. Email: pld43@drexel.edu} and Hugo J. Woerdeman\textsuperscript{a}\thanks{H.~J. Woerdeman. Email: hugo@math.drexel.edu}}
\affil{\textsuperscript{a} Department of Mathematics, Drexel University, Philadelphia, USA; \textsuperscript{b} Institute of Mathematics, University of the Philippines Diliman, Quezon City, Philippines}}

\maketitle

\begin{abstract}
Let $\mu_1$ be a complex number in the numerical range $W(A)$ of a normal matrix $A$. In the case when no eigenvalues of $A$ lie in the interior of $W(A)$, we identify the smallest convex region containing all possible complex numbers $\mu_2$ for which $\begin{bmatrix}\mu_1& *\\0& \mu_2\end{bmatrix}$ is a $2$-by-$2$ compression of $A$.\end{abstract}
\begin{amscode}
15A18, 15A29,
15A60, 47A12,	47A20 
\end{amscode}
\begin{keywords}
Eigenvalues; Ritz values; Normal matrices; Interlacing; Matrix compressions
\end{keywords}

\section{Introduction}
Let $k,n\in\mathbb N$ with $k\leq n$. A matrix $B\in\C^{k\times k}$ is a \textit{size}-$k$ \textit{compression} of $A\in\C^{n\times n}$ if there exists a unitary $U\in\C^{n\times n}$ such that 
\begin{equation}\label{intro1}
U^*AU=\begin{bmatrix} B& *\\ *& *\end{bmatrix}.\end{equation}
Equivalently, $B$ is a size-$k$ compression of $A$ if there exists an isometry $V\in \C^{n\times k}$ (i.e., $V^*V=I_k$) such that $B=V^*AV$. If $B$ is a size-$k$ compression of $A$ that is a scalar matrix or a normal matrix, then $B$ is called a \textit{scalar} or \textit{normal} compression of $A$, respectively. The eigenvalues of the compression $V^*AV$ are called the \textit{Ritz values from the pair} $(A,V)$, and in this case, we say that the Ritz values form a $k$-\textit{Ritz set of} $A$. 
The Ritz values of $A$ are inside the \textit{numerical range} or \textit{field of values} $W(A)$ of $A$ which is defined by
\[W(A):=\{x^*Ax:\ x\in\C^n,\ ||x||=1\}.\]
Characterizing which $k$ complex numbers $\mu_1,\ldots,\mu_k\in W(A)$ appear as Ritz values of $A$ is called the \textit{inverse field of values problem with $k$ Ritz values} (iFOV-$k$) \cite{fU08,rC09,rC11,rC13}. More precisely, the iFOV-$k$ problem for $A$ is described as follows: given $\mu_1,\ldots,\mu_k\in W(A)$, (i) determine if there is an isometry $V\in\C^{n\times k}$ such that $\sigma(V^*AV)=\{\mu_1,\ldots,\mu_k\}$, and if so, (ii) characterize such $V$. Uhlig originally proposed the iFOV-$1$ problem in \cite{fU08} which is always solvable while Carden considered the generalization to iFOV-$k$ \cite{rC11,rC13}. Ritz values have been studied in numerical analysis in the setting of Krylov subspace methods (see, e.g.,  \cite{mE03,mE09,jD12}). Several papers deal with characterizing $(n-1)$-Ritz sets (see, e.g., \cite{hG02,hG04,sM04,rC13}). The known results on $k$-Ritz sets for $k<n-1$ usually assume that the size-$k$ compression has some special properties (see, e.g., \cite{kF57,dC84,jQ09,jH12,zB13}) or the results are algebraically formulated but have not been given geometric interpretation \cite{rT66,sM04}. We begin to carry out the program of finding a geometric characterization of $k$-Ritz sets in general, and our main result on $2$-Ritz sets Theorem \ref{main} suggests a promising starting point. A paper that inspired the current research is by Carden and Hansen \cite{rC13} where they proved that if $A\in\C^{3\times 3}$ is normal whose boundary $\partial W(A)$ forms a nondegenerate triangle, then fixing $\mu_1\in W(A)\setminus\partial W(A)$ determines a unique number $\mu_2$ (called the \textit{isogonal conjugate of} $\mu_1$) so that $\{\mu_1,\mu_2\}$ forms a $2$-Ritz set for $A$ \cite[Theorem 4]{rC13}. In this paper, we consider the analogous problem for an $n$-by-$n$ normal matrix $A$, and we identify the smallest convex region containing all possible $\mu_2$'s.

Let $A\in \C^{n\times n}$ be normal and let $\mu_1\in W(A)$ be given. Generalizing the notation in Section 3 of \cite{jH12}, we define the set of all Ritz values associated to $\mu_1$ as
\[
{\cal B}_A(\mu_1):=\left\{\mu_2\in\C:\ V^*A V=\begin{bmatrix} \mu_1&*\\0& \mu_2\end{bmatrix}\ \textup{for some}\ V\in\C^{n\times 2}\ \textup{with}\ V^*V=I_2\right\}.
\] Let $K$ be a $3$-element subset of the \textit{spectrum} $\sigma(A)$ of $A$. 
If $\mu_1\notin\sigma(A)$, define \[{\cal E}_K(\mu_1):=\begin{cases} \{w_{K}\}\cup[\sigma(A)\ominus K],& \textup{if}\ \mu_1\in\textup{conv}(K)\\
												\varnothing,& \textup{otherwise},\end{cases}\]
where $w_K$ is the \textit{isogonal conjugate} of $\mu_1$ with respect to $\textup{conv}(K)$ (to be defined in Section 2) and the set difference $\sigma(A)\ominus K$ is to be taken according to multiplicity. Thus, if $\lambda\in \sigma(A)$ has multipicity $m$, and $\lambda$ appears $k$ times in $K$, then $\lambda$ appears $m-k$ times in $\sigma(A)\ominus K$. If $\Lambda_2(A)$ denotes the \textit{rank-2 numerical range of} $A$ (see Section 3 for the definition and some properties), define
\small
\begin{equation}\label{region}
{\cal R}_A(\mu_1):=\begin{cases}
\vspace{0.2cm} 
W(A),& \textup{if}\ \mu_1\in \Lambda_2(A)\\
\vspace{0.2cm} 
\textup{conv}[\sigma(A)\ominus\{\lambda_i\}],& \textup{if}\ \mu_1=\lambda_i,\ \textup{where}\ \lambda_i\notin \Lambda_2(A)\ \textup{or}\\
\vspace{0.2cm} 
& \mu_1\in [W(A)\setminus\Lambda_2(A)]\cap (\lambda_i,\lambda_{i+1}),\\
\vspace{0.2cm} 
& \textup{where}\ [\lambda_i,\lambda_{i+1}]\cap \partial W(A)=[\lambda_i,\lambda_{i+1}]\\
\vspace{0.2cm} 
& \textup{and}\ \lambda_i\notin\Lambda_2(A),\ \lambda_{i+1}\in \Lambda_2(A)\\
\vspace{0.2cm} 
\textup{conv}[\sigma(A)\ominus\{\lambda_i,\lambda_{i+1}\}],& \textup{if}\ \mu_1\in [W(A)\setminus\Lambda_2(A)]\cap (\lambda_i,\lambda_{i+1}),\\
\vspace{0.2cm} 
& \textup{where}\ [\lambda_i,\lambda_{i+1}]\cap \partial W(A)=[\lambda_i,\lambda_{i+1}]\\
\vspace{0.2cm} 
&\textup{and}\ \lambda_i,\lambda_{i+1}\notin\Lambda_2(A)\\
\vspace{0.2cm} 

\textup{conv}\left[\displaystyle\bigcup_{K\subseteq \sigma(A), |K|=3}{\cal E}_K(\mu_1)\right],&
 \textup{if}\ \mu_1\in [W(A)\setminus(\Lambda_2(A)\cup\partial W(A))]. 
\end{cases}
\end{equation}
The \textit{generating points of} ${\cal R}_A(\mu_1)$ are either the extreme points of ${\cal R}_A(\mu_1)$ if $\mu_1$ satisfies the first three cases above or the elements of $\displaystyle\bigcup_{K\subseteq \sigma(A), |K|=3}{\cal E}_K(\mu_1)$ otherwise. Observe that all the generating points of ${\cal R}_A(\mu_1)$ are elements of ${\cal B}_A(\mu_1)$ (see \cite[Proposition 6]{jH12} and \cite[Theorem 4]{rC13}).
\normalsize

For a normal $A\in \C^{4\times 4}$ with a $4$-gon for its $\partial W(A)$, Carden and Hansen noted without proof in \cite[Figure 6]{rC13} that once an interior point $\mu_1\in W(A)$ is fixed, then $\textup{conv}[{\cal B}_A(\mu_1)]={\cal R}_A(\mu_1)=\textup{conv}\{w_{123},w_{124}, \lambda_3,\lambda_4\}$. (See Figure 1.)
\begin{figure}[h]
\centering
\includegraphics[height=7.5cm, width=10.5cm]{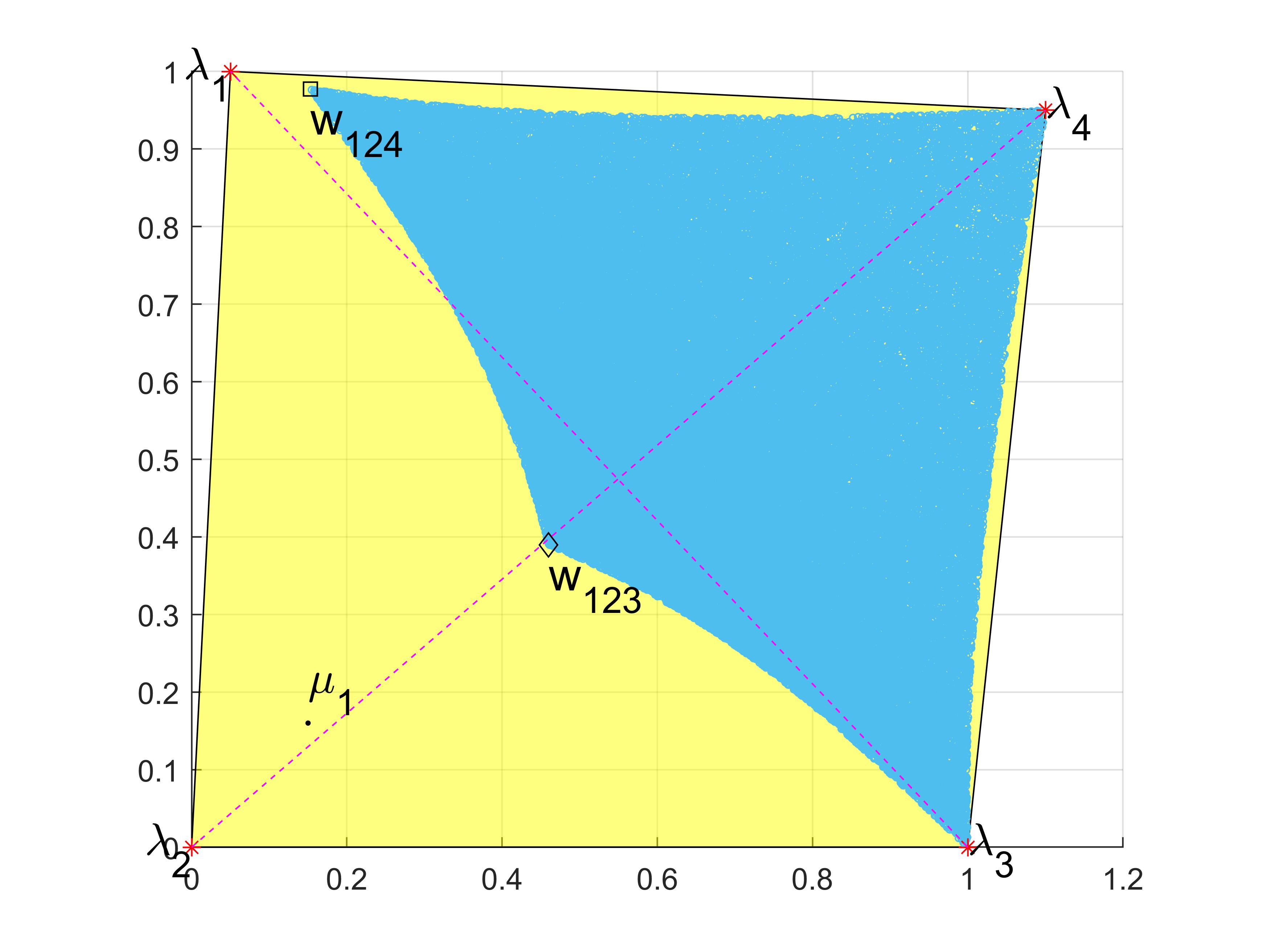}
\caption{A version of Figure 6 in \cite{rC13}. The blue region is the set ${\cal B}_A(\mu_1)$.}
\end{figure}

The goal of the paper is to confirm Carden and Hansen's observation in general, that is, $\textup{conv}[{\cal B}_A(\mu_1)]={\cal R}_A(\mu_1)$. In our main result, we exclude the case when the eigenvalues of $A$ all lie on the same line, as in this case ${\cal B}_A(\mu_1)$ is well understood due to Cauchy interlacing.

\begin{theorem}\label{main}
Let $A\in\C^{n\times n}$ be normal with eigenvalues $\lambda_1,\ldots,\lambda_n$ not lying on the same line and arranged in a counterclockwise orientation with respect to $\textup{trace}(A)/n$ such that no eigenvalue is in the interior of $W(A)$. If $\mu_1\in W(A)$, then
\[\textup{conv}[{\cal B}_A(\mu_1)]= {\cal R}_A(\mu_1)\]
where ${\cal R}_A(\mu_1)$ is as defined in \eqref{region}. 
\end{theorem}

Let $\mu_1=\displaystyle\sum_{j=1}^nt_j\lambda_j$ where $t_j>0$ and $\displaystyle\sum_{j=1}^nt_j=1$. Let $w=[\lambda_3\ \cdots\ \lambda_n]^T$ and $\mathbbm 1=[1\ \cdots\ 1]^T\in \C^{n-2}$. An important technique in our proof is that under appropriate conditions, Theorem \ref{main} follows from showing that the matrix $Z(t)$ defined by
\footnotesize
\begin{equation}\label{psd}
\dfrac{\textup{Re}(\lambda_1)}{t_1}(w-\lambda_2 \mathbbm 1)(w-\lambda_2 \mathbbm 1)^*+\dfrac{\textup{Re}(\lambda_2)}{t_2}(\lambda_1\mathbbm 1-w)(\lambda_1\mathbbm 1-w)^*+|\lambda_2-\lambda_1|^2\textup{diag}\left(\dfrac{\textup{Re}(\lambda_3)}{t_3},\ldots,\dfrac{\textup{Re}(\lambda_n)}{t_n}\right)\end{equation}
\normalsize
is positive semidefinite.

\section{Preliminaries}
We begin this section by defining the \textit{isogonal conjugate of a point} with respect to some triangular region.
\begin{definition}
Let $\lambda_a,\lambda_b,\lambda_c\in\C$ be successive corners of a nondegenerate triangular region $\textup{conv}\{\lambda_a,\lambda_b,\lambda_c\}$ having counterclockwise orientation with respect to $\sum_j\lambda_j/3$. Given a point $\mu\in\textup{conv}\{\lambda_a,\lambda_b,\lambda_c\}\setminus\{\lambda_a,\lambda_b,\lambda_c\}$, let $\ell_j$ be the line through $[\lambda_j,\mu]$ and ${\cal L}_j$ be the reflection of $\ell_j$ about the angle bisector of the vertex angle at $\lambda_j$. The \textit{isogonal conjugate} $w_{abc}$ of $\mu$ is the intersection of ${\cal L}_j$ for all $j=a,b,c$.
\end{definition}

Let $X\in \mathbb C^{m\times n}$ be given and let $\alpha\subseteq \{1,\ldots,m\}$ and $\beta\subseteq\{1,\ldots,n\}$ be arbitrary index sets of cardinality $0<r\leq \min\{m,n\}$. The $r^{th}$ \textit{compound matrix of} $X$, denoted by $C_r(X)$, is the ${{m}\choose{r}}\times  {{n}\choose{r}}$ matrix whose $(\alpha,\beta)$-entry is the determinant $\det (X[\alpha,\beta])$ of the submatrix of $X$ whose rows and columns are indexed by $\alpha $ and $\beta$, respectively, and the indexing is lexicographic. By Cauchy-Binet formula, the $r^{th}$ compound matrix satisfies \cite[Equation 0.8.1.1]{rH13} \[C_r(XY)=C_r(X)C_r(Y).\]

Let $A\in \C^{n\times n}$ be normal with eigenvalues $\lambda_1,\ldots,\lambda_n$. If $B$ is a size-$k$ compression of $A$, then there exists an isometry $U\in\C^{n\times k}$ such that
\[\left(C_j(U)\right)^*C_j(A-z I_n)\left(C_j(U)\right)=C_j(B-z I_k)\]
for all $z\in\C$ and for each $j=1,\ldots,k$ (see \cite{rT66,sM04}). This follows from the multiplicative property of the compound matrix. In particular, when $j=k=2$ and the spectrum $\sigma(B)=\{\mu_1,\mu_2\}$, then
\begin{equation}\label{nec2}(\mu_1-z)(\mu_2-z)\in \textup{conv}\{(\lambda_i-z)(\lambda_j-z):\ i<j\}\end{equation}
for all $z\in\C$. 

Numerical experiments reveal that \eqref{nec2} is not sufficient. We extend the ideas in \cite{jH12} to gain a better understanding of ${\cal B}_A(\mu_1)$. Schur's Theorem \cite[Theorem 2.3.1]{rH13} guarantees that $A$ is unitarily similar to $\Lambda:=\textup{diag}(\lambda_1,\ldots,\lambda_n)$. Observe that ${\cal B}_{A}(\mu_1)={\cal B}_{\Lambda}(\mu_1)$. Define the collection of all convex weights of $\mu_1$ as
\[
{\cal C}_\Lambda(\mu_1):=\left\{t=[t_j]\in\R^n:\ \displaystyle\sum_{j=1}^nt_j\lambda_j=\mu_1, \displaystyle\sum_{j=1}^nt_j=1,\ t_j\geq 0\right\}.
\]

If $\mu_1=\langle\Lambda u,u\rangle$ for some unit vector $u\in\C^n$, then without loss of generality, we can assume $u$ has nonnegative entries since $\Lambda$ is diagonal. The set of all possible Ritz vector $u\in\C^n$ of $\mu_1$ is given by the set $\{\sqrt{t}:\ t\in {\cal C}_\Lambda(\mu_1)\}$
where $\sqrt{t}$ is componentwise square root. For each $t\in {\cal C}_\Lambda(\mu_1)$, there exists $F_t\in \C^{n\times(n-m)}$ with $F_t^*F_t=I_{n-m}$ whose range space is $\{\sqrt{t},\Lambda\sqrt{t}\}^\perp$. Since $t\neq 0$, the dimension $m$ of $\textup{Span}(\{\sqrt{t},\Lambda\sqrt{t}\})$ satisfies $0<m\leq 2$. If $m=1$, then $\mu_1\in \sigma(A)$. 

Define
\[
{\cal B}_\Lambda(\mu_1,t):=W(F_t^*\Lambda F_t).
\]
In the normal compression case, the next result reduces to \cite[Proposition 10]{jH12}.

\begin{proposition}\label{givenzprop1}
Let $A\in \C^{n\times n}$ be normal, and let $\mu_1\in W(A)$ be given. Then
\[{\cal B}_A(\mu_1)=\displaystyle\bigcup_{t\in{\cal C}_\Lambda(\mu_1)}{\cal B}_\Lambda(\mu_1,t).\]
\end{proposition}

Under the assumption that no three eigenvalues lie on a line,
\cite[Proposition 11]{jH12} guarantees that
\[\textup{Ext}[{\cal C}_\Lambda(\mu_1)]=\{t\in {\cal C}_\Lambda(\mu_1):\ t\ \textup{has at most 3 positive entries}\}.\] If some but not all eigenvalues lie on a line, then a similar proof to \cite[Proposition 11]{jH12} implies the containment
\begin{equation}\label{ext}
\textup{Ext}[{\cal C}_\Lambda(\mu_1)]\subseteq\{t\in {\cal C}_\Lambda(\mu_1):\ t\ \textup{has at most 3 positive entries}\}.
\end{equation}
The proof in \cite{jH12} of the reverse inclusion of \eqref{ext} hinges on the \textit{uniqueness} of the weights when writing $\mu_1$ as a convex combination of eigenvalues $\lambda_i,\lambda_j,\lambda_k$ which are assumed to be not on the same line in \cite{jH12}. The convex weights of $\mu_1$ are non-unique when the eigenvalues $\lambda_i,\lambda_j,\lambda_k$ lie on a line. However, if $\mu_1\notin[\lambda_i,\lambda_j]$ for all $i,j$, then equality in \eqref{ext} is attained and the extreme points can be taken to be those elements of ${\cal C}_\Lambda(\mu_1)$ with exactly three positive entries.

\begin{proposition}\label{givenzprop2}
Let $A\in\C^{n\times n}$ be normal with eigenvalues $\lambda_1,\ldots,\lambda_n$, and let $\mu_1\in W(A)$ be given. The set ${\cal C}_\Lambda(\mu_1)$ is compact and convex.
Moreover, if the eigenvalues $\lambda_1,\ldots,\lambda_n$ are not on lying on the same line, then \[\textup{Ext}[{\cal C}_\Lambda(\mu_1)]\subseteq\{t\in {\cal C}_\Lambda(\mu_1):\ t\ \textup{has at most 3 positive entries}\}.\]
In particular, if $\mu_1\notin[\lambda_i,\lambda_j]$ for all $i,j$, then \[{\cal C}_\Lambda(\mu_1)=\textup{conv}\{t\in{\cal C}_\Lambda(\mu_1):\ t\ has\ exactly\ 3\ positive\ entries\}.\]
\end{proposition}
The next result gives a formula for the isogonal conjugate of a point with respect to some triangular region.
\begin{proposition}\label{isogformula1}
Let $z_1,z_2,z_3\in\C$ be successive corners of a nondegenerate triangular region $ \textup{conv}\{z_1,z_2,z_3\}$ having counterclockwise orientation with respect to $\sum_{j=1}^3z_j/3$. Given $\mu=r_1z_1+r_2z_2+r_3z_3$ where $r_i>0$ and $\displaystyle\sum_{i=1}^3r_i=1$, let $f:=\dfrac{\overline{z_3-z_2}}{\sqrt{r_1}}e_1+\dfrac{\overline{z_1-z_3}}{\sqrt{r_2}}e_2+\dfrac{\overline{z_2-z_1}}{\sqrt{r_3}}e_3$ where $e_j$ is the $j^{th}$ standard basis vector in $\C^3$. Then the isogonal conjugate $w_{123}$ of $\mu$ is given by $w_{123}=\dfrac{f^*\Lambda f}{f^*f}$.
\end{proposition}
\begin{proof}
Let $A:=\textup{diag}(z_1,z_2,z_3)$ and $r:=[r_1\ r_2\ r_3]^T$. Then a calculation reveals that $f\in \{\sqrt{r},A\sqrt{r}\}^\perp.$
Observe that ${\cal C}_A(\mu)=\{r\}$ due to the uniqueness of writing $\mu$ as a convex combination of the $z_k$'s. By Proposition \ref{givenzprop1},
${\cal B}_A(\mu)={\cal B}_A(\mu,r)=\{w\}$
where $w=\dfrac{f^*Af}{f^*f}$. By \cite[Theorem 4]{rC13}, $w$ is the isogonal conjugate of $\mu$.
\end{proof}

\begin{proposition}\label{withzero}
Let $A\in \C^{n\times n}$ be normal with eigenvalues $\lambda_1,\ldots,\lambda_n$ and $\mu_1\in W(A)$. Let $t\in {\cal C}_\Lambda(\mu_1)$, $J:=\{j: t_j=0\}\neq\varnothing$, and $S:=\{\lambda_j:j\in J\}$. Then \[{\cal B}_\Lambda(\mu_1,t)=\textup{conv}[{\cal B}_{\Lambda'}(\mu_1,s)\cup W(S)]\]
where $\Lambda'=\textup{diag}(\lambda_j)_{j\notin J}$ and $s=[t_j]_{j\notin J}$.
\end{proposition}
\begin{proof}
Assume $\mu_1\in W(A)\setminus \sigma(A)$. The proof for $\mu_1\in \sigma(A)$ is analagous. Observe that $e_j\in \{\sqrt{t},\Lambda \sqrt{t}\}^\perp$ for all $j\in J$, where $e_j$ is the $j^{th}$ standard basis vector in $\C^n$. By applying a permutation similarity, we can assume $J=\{1,\ldots,m\}$. Write $t=[\overbrace{0\ \cdots\ 0}^m\ s^T]^T$ for some $s\in {\cal C}_{\Lambda'}(\mu_1)$ where $\Lambda'=\textup{diag}(\lambda_{m+1},\ldots,\lambda_n)$. Note that $\{ e_1, \cdots, e_m\}\subseteq \{\sqrt{t},\Lambda\sqrt{t}\}^\perp$ which can be extended to an orthonormal basis $\{f_1,\ldots,f_{n-2-m}, e_1, \cdots, e_m\}$ of $\{\sqrt{t},\Lambda\sqrt{t}\}^\perp$. Set 
\[F_t:=[f_1\ \cdots\ f_{n-2-m}\ e_1\ \cdots\ e_m]\in \C^{n\times(n-2)}.\]
Observe that $f_j=[\overbrace{0\ \cdots\ 0}^m\ g_j^T]^T$ and $g_j\in \{\sqrt{s}, \Lambda'\sqrt{s}\}^\perp$. Moreover, $G_s:=[g_1\ \cdots\ g_{n-2-m}]\in\C^{n\times (n-2-m)}$ satisfies $G_s^*G_s=I_{n-2-m}$ and 
\[F_t^*\Lambda F_t=G_s^*\Lambda'G_s\oplus\textup{diag}(\lambda_1,\ldots,\lambda_m)\]
where $\Lambda'=\textup{diag}(\lambda_{m+1},\ldots,\lambda_n)$. It follows that
\[\begin{array}{rcl}
{\cal B}_\Lambda(\mu_1,t)&=&W(F_t^*\Lambda F_t)\\
&=&\textup{conv}[W(G_s^*\Lambda'G_s)\cup\textup{conv}\{\lambda_1,\ldots,\lambda_m\}]\\
&=&\textup{conv}[{\cal B}_{\Lambda'}(\mu_1,s)\cup W(S)].
\end{array}\]
\end{proof}

\section{Cases for $\mu_1$}
To prove Theorem \ref{main}, we consider cases depending on whether or not $\mu_1$ is in the \textit{rank-2 numerical range}. Choi, Kribs, and \.Zyczkowski proposed a ``compression formalism'' approach to solve the quantum error correction problem \cite{mc06, mC062}, and this led them to define the \textit{rank}-$k$ \textit{numerical range} of $A\in\C^{n\times n}$ as
\[
\Lambda_k(A):=\{\lambda\in\C: PAP=\lambda P,\ \textup{for some}\ \textup{rank-}k\ \textup{orthoprojection}\ P\}.\]
Note that
$\lambda\in \Lambda_k(A)$ if and only if $\lambda I_k$ is a scalar compression of $A$. It is shown in \cite{cL07} that $\Lambda_k(A)\neq\varnothing$ if $k<n/3+1$, and as a consequence, $\Lambda_2(A)\neq \varnothing$ if $n\geq 4$. The classical numerical range of $A$ is $W(A)=\Lambda_1(A)$ which is convex by the Toeplitz-Hausdorff Theorem. It turns out that in general, $\Lambda_k(A)$ is convex as established in \cite{cL08,hW08}.

If $A$ is normal, then the main result in \cite{cL08} implies
\begin{equation}\label{gen1}\Lambda_2(A)=\displaystyle\bigcap_{\lambda\in \sigma(A)}\textup{conv}[\sigma(A)\ominus \{\lambda\}]\end{equation}
(where $\ominus$ denotes set difference counting multiplicities) confirming Conjecture 2.8 in \cite{mc06}.

Let $A\in\C^{n\times n}$ be normal with eigenvalues $\lambda_1,\ldots,\lambda_n$ such that no $\lambda_j$ is in the interior of $W(A)$. Suppose $\mu_1\in W(A)\setminus [\Lambda_2(A)\cup\partial W(A)]$. By \eqref{gen1}, there exists $\lambda_a\in \sigma(A)$ such that
\[\mu_1\notin \textup{conv}[\sigma(A)\ominus\{\lambda_a\}].\]
Since $\mu_1\in W(A)$, $\lambda_a\notin \Lambda_2(A)$, and necessarily $\lambda_a$ has multiplicity $1$. Moreover, there exist $\lambda_{a-1},\lambda_{a+1}\in\sigma(A)$ (modulo $n$) for which $\mu_1\in\textup{conv}\{\lambda_{a-1},\lambda_a,\lambda_{a+1}\}$. Note that the eigenvalues are necessarily consecutive in the boundary of $W(A)$.

If there exists another $\lambda_b\neq \lambda_a$ with the property
\[\mu_1\notin \textup{conv}[\sigma(A)\ominus\{\lambda_b\}],\]
then we claim that $\lambda_b=\lambda_{a-1}$ or $\lambda_b=\lambda_{a+1}$. Otherwise,
\[\mu_1\in\textup{conv}\{\lambda_{a-1},\lambda_a,\lambda_{a+1}\}\subseteq\textup{conv}[\sigma(A)\setminus\{\lambda_b\}]\subseteq \textup{conv}[\sigma(A)\ominus\{\lambda_b\}],\]
a contradiction. Without loss of generality, assume $\lambda_b=\lambda_{a+1}$ which satisfies $\lambda_{a+1}\notin \Lambda_2(A)$, and so $\lambda_{a+1}$ has multiplicity $1$ necessarily. Since $\mu_1\notin \textup{conv}[\sigma(A)\ominus\{\lambda_{a+1}\}]$ by assumption, $\mu_1\in\textup{conv}\{\lambda_a,\lambda_{a+1},\lambda_{a+2}\}$. Hence, $\mu_1\in\textup{conv}\{\lambda_{a-1},\lambda_a,\lambda_{a+1}\}\cap\textup{conv}\{\lambda_a,\lambda_{a+1},\lambda_{a+2}\}
$. 

Thus, we consider the following cases regarding $\mu_1$:\\
\noindent\textbf{Case 1}: $\mu_1\in \Lambda_2(A)$.\\

\noindent\textbf{Case 2}: $\mu_1\in \partial W(A)\cap [W(A)\setminus\Lambda_2(A)]$.\\

\noindent\textbf{Case 3}: $\mu_1\in W(A)\setminus [ \Lambda_2(A)\cup \partial W(A)]$ and there exists unique $\lambda_a\in \sigma(A)\setminus \Lambda_2(A)$ for which
\[\mu_1\notin \textup{conv}[\sigma(A)\ominus\{\lambda_a\}].\]

\noindent\textbf{Case 4}: $\mu_1\in W(A)\setminus [ \Lambda_2(A)\cup \partial W(A)]$ and there exist $\lambda_a,\lambda_{a+1}\in \sigma(A)\setminus \Lambda_2(A)$ for which
\[\mu_1\in\textup{conv}\{\lambda_{a-1},\lambda_a,\lambda_{a+1}\}\cap\textup{conv}\{\lambda_a,\lambda_{a+1},\lambda_{a+2}\}
\] and the intersection is of two nondegenerate triangular regions. 

We start by proving \textbf{Case 1} of Theorem \ref{main}.

\begin{proposition}\label{inrank2}
Let $A\in\C^{n\times n}$ be normal and $\mu_1\in \Lambda_2(A)$. Then 
\[\textup{conv}[{\cal B}_A(\mu_1)]={\cal R}_A(\mu_1).\]
\end{proposition}
\begin{proof}
For each $\lambda_i\in \sigma(A)$, $\mu_1\in \textup{conv}[\sigma(A)\ominus \{\lambda_i\}]$. Then $\lambda_i\in{\cal B}_A(\mu_1)$ due to Proposition \ref{withzero}, and so $\textup{conv}[{\cal B}_A(\mu_1)]=W(A)={\cal R}_A(\mu_1)$.
\end{proof}

Next, we will use the following observation to prove \textbf{Case 2}, which will be the content of Proposition \ref{twocases}.

\begin{lemma}\label{corner}
Let $A\in \C^{n\times n}$ be normal with eigenvalues $\lambda_1,\ldots,\lambda_n$ such that no eigenvalue is in the interior $W(A)$. If $\lambda_i\in \sigma(A)\setminus \Lambda_2(A)$, then $\lambda_i$ is a corner of $W(A)$.
\end{lemma}
\begin{proof}
Suppose $\lambda_i$ is not a corner. Since no eigenvalue is in the interior of $W(A)$, there exist corners $\lambda_a, \lambda_b$ such that $\lambda_i\in (\lambda_a,\lambda_b)\subseteq\partial W(A)$. It follows that \[\lambda_i\in [\lambda_a,\lambda_b]\subseteq \textup{conv}[\sigma(A)\ominus \{\lambda_i\}].\]
This implies $\lambda_i\in \Lambda_2(A)$.
\end{proof}

Let $A\in\C^{n\times n}$ be normal satisfying the assumptions of Theorem \ref{main}. Suppose $\mu_1\in \partial W(A)\cap[W(A)\setminus \Lambda_2(A)]$. Then $\mu_1\in [\lambda_i,\lambda_j]$ where $1\leq i<j\leq n$. Due to the assumption that $\mu_1\notin\Lambda_2(A)$ and the convexity of $\Lambda_2(A)$, we can assume one endpoint is not in $\Lambda_2(A)$, say $\lambda_i\notin\Lambda_2(A)$. If there are only two eigenvalues, namely, $\lambda_i$ and $\lambda_j$ on the line through $[\lambda_i,\lambda_j]$, then $\lambda_j=\lambda_{i+1}$. Otherwise, $[\lambda_i,\lambda_j]\subseteq [\lambda_a,\lambda_b]$ where $a<b$. If both $\lambda_i,\lambda_j\notin\{\lambda_a,\lambda_b\}$, then $\mu_1\in[\lambda_i,\lambda_j]\subseteq\Lambda_2(A)$, a contradiction. Without loss of generality, we assume $\lambda_i=\lambda_a$ and $\lambda_j\neq \lambda_b$. It follows that $\lambda_j\in \Lambda_2(A)$. If there is another $\lambda_k\in (\lambda_i,\lambda_j)$ such that $\mu_1\in [\lambda_j,\lambda_k]$, then $\lambda_k\in \Lambda_2(A)$ which implies $\mu_1\in [\lambda_j,\lambda_k]\subseteq\Lambda_2(A)$, a contradiction. Thus, $\lambda_j=\lambda_{i+1}$.

\begin{proposition}\label{twocases}
Given a normal matrix $A\in\C^{n\times n}$ satisfying the conditions of Theorem \ref{main}, let $\mu_1\in \partial W(A)\cap [W(A)\setminus\Lambda_2(A)]$. Then $\textup{conv}[{\cal B}_A(\mu_1)]={\cal R}_A(\mu_1)$.
\end{proposition}
\begin{proof}
Suppose $\mu_1=\lambda_i$, where $\lambda_i\notin \Lambda_2(A)$. By Lemma \ref{corner}, $\mu_1$ is a corner and necessarily has multiplicity $1$. It follows that ${\cal C}_\Lambda(\mu_1)=\{e_i\}$, where $e_i$ is the $i^{th}$ standard basis vector in $\C^n$. By Proposition \ref{givenzprop1},
\[{\cal B}_A(\mu_1)={\cal B}_A(\mu_1,e_i)=\textup{conv}[\sigma(A)\ominus \{\lambda_i\}]={\cal R}_A(\mu_1).\]

If $\mu_1\in \partial W(A)\cap [W(A)\setminus\Lambda_2(A)]$, then $\mu_1\in [\lambda_i,\lambda_{i+1}]$ and without loss of generality, assume $\lambda_i\notin\Lambda_2(A)$. If $\mu_1\notin\sigma(A)$, then $\mu_1\in (\lambda_i,\lambda_{i+1})$. Due to Lemma \ref{corner}, $\lambda_i$ is a corner of $W(A)$. Consider two cases on whether or not $\lambda_{i+1}$ is in $ \Lambda_2(A)$. If $\lambda_{i+1}\notin \Lambda_2(A)$, then $\lambda_{i+1} $ is also a corner of $W(A)$ by Lemma \ref{corner}. Since $(\lambda_i,\lambda_{i+1})\subseteq \partial W(A)$ and the endpoints are corners each with multiplicity $1$, ${\cal C}_\Lambda(\mu_1)=\{v\}$, where $v=\sqrt{1-p}e_i+\sqrt{p}e_{i+1}$ for some $p\in (0,1)$. By Proposition \ref{givenzprop1},
\[{\cal B}_A(\mu_1)={\cal B}_A(\mu_1,v)=\textup{conv}[\sigma(A)\ominus \{\lambda_i,\lambda_{i+1}\}]={\cal R}_A(\mu_1).\]
Finally, if $\lambda_{i+1}\in \Lambda_2(A)$, let $[\lambda_i,\lambda_m]\subseteq \partial W(A)$ be the largest line segment on $\partial W(A)$ containing $\mu_1$. By Cauchy interlacing, ${\cal B}_{\Lambda'}(\mu_1)=[\lambda_{i+1},\lambda_m]$ where $\Lambda'=\textup{diag}(\lambda_i,\lambda_{i+1},\ldots,\lambda_m)$. Moreover, Cauchy interlacing and Proposition \ref{withzero} guarantee $\sigma(A)\ominus \{\lambda_i\}\subseteq {\cal B}_A(\mu_1)$, and so ${\cal R}_A(\mu_1)\subseteq {\cal B}_A(\mu_1)$. For the reverse inclusion, observe that any $t\in {\cal C}_\Lambda(\mu_1)$ has zero entries on the $j^{th}$ position where $j\in J:=\{1,\ldots,i-1,m+1,\ldots,n\}$. Let $S:=\{\lambda_j:j\in J\}$. By Propositions \ref{givenzprop1} and \ref{withzero},
\[
\begin{array}{rcl}{\cal B}_A(\mu_1)=\displaystyle\bigcup_{t\in {\cal C}_\Lambda(\mu_1)}{\cal B}_\Lambda(\mu_1,t)
&=&\displaystyle\bigcup_{s\in {\cal C}_{\Lambda'}(\mu_1)}\textup{conv}[{\cal B}_{\Lambda'}(\mu_1,s)\cup W(S)]\\
&\subseteq& \displaystyle\bigcup_{s\in {\cal C}_{\Lambda'}(\mu_1)}\textup{conv}[ [\lambda_{i+1},\lambda_m]\cup W(S)]\\
&=&\textup{conv}[\sigma(A)\ominus \{\lambda_i\}]={\cal R}_A(\mu_1).
\end{array}\]\end{proof}

As we will see in Corollary \ref{distinct}, it suffices to consider normal matrices with distinct eigenvalues. This will greatly simplify the notations and arguments in the remainder of this section and in Sections 4-5.

\begin{proposition}\label{multip}
Given $D=\bigoplus_{i=1}^n\lambda_i I_{k_i}\in \C^{m\times m}$ where $m=\displaystyle\sum_{i=1}^nk_i$, the multiplicities $k_1,\ldots,k_n\geq 1$, and the eigenvalues $\lambda_1,\ldots,\lambda_n$ are distinct, let $\Lambda:=\textup{diag}(\lambda_1,\ldots,\lambda_n)$, $\Lambda':=\bigoplus_{i=1}^n\lambda_i I_{k_i-1}$, and $\mu_1\in W(D)$. If $\mu_2\in {\cal B}_D(\mu_1)$, then there exists an isometry $V\in \C^{m\times 2}$ having the form $V=\begin{bmatrix} v_{11}& v_{12}\\ 0& v_{22}\end{bmatrix}$ where $v_{11},v_{12}\in \C^n$ such that $\begin{bmatrix} \mu_1& *\\ 0& \mu_2\end{bmatrix}=V^*D V$.
\end{proposition}
\begin{proof}
If $\mu_2\in {\cal B}_D(\mu_1)$, then there exists an isometry $W\in \C^{m\times 2}$ with $\begin{bmatrix} \mu_1& *\\ 0& \mu_2\end{bmatrix}=W^*D W$. Write $W=\begin{bmatrix} w_{11}& w_{12}\\ \vdots& \vdots\\ w_{n1}& w_{n2}\end{bmatrix}$ where $w_{ij}\in \C^{k_i}$ for $i=1,\ldots,n$. For each $i$, there exists unitary $U_i\in \C^{k_i\times k_i}$ such that \[U_iw_{i1}=||w_{i1}||e_1^{(k_i)}\]
where $e_p=e_p^{(q)}$ is the $p^{th}$ standard basis vector in $\C^q$. The matrix $U:=\bigoplus_{i=1}^n U_i\in \C^{m\times m}$ is unitary, and it satisfies $W^*DW=W^*U^*DUW$ since $U$ and $D$ commute. There exists permutation matrix $P\in \R^{m\times m}$ such that $V:=PUW$ has the desired form and $PDP^T=\Lambda\oplus \Lambda'$. Indeed, consider $P=\begin{bmatrix} P_1\\ P_2\end{bmatrix}$ where $P_1\in\C^{n\times m}$ is defined as $P_1^T=\begin{bmatrix} e_1& e_{k_1+1}& \cdots&  e_{k_1+\cdots+k_{n-1}+1}\end{bmatrix}$ and $P_2^T=\begin{bmatrix} e_2& \cdots& e_{k_1}&e_{k_1+2}&\cdots& e_{k_1+k_2}&\cdots& e_{k_1+\cdots+k_{n-1}+2}&\cdots& e_{k_1+\cdots+k_n}\end{bmatrix}$.
\end{proof}

\begin{corollary}\label{distinct}
Assuming the conditions in Proposition \ref{multip}, ${\cal B}_D(\mu_1)=\textup{conv}[{\cal B}_\Lambda(\mu_1)\cup W(\Lambda')]$. Moreover, if $\textup{conv}[{\cal B}_\Lambda(\mu_1)]={\cal R}_\Lambda(\mu_1)$, then $\textup{conv}[{\cal B}_D(\mu_1)]={\cal R}_D(\mu_1)$.
\end{corollary}
\begin{proof}
For the first part, it suffices to prove the $\subseteq$ inclusion due to Proposition \ref{withzero}. Let $\mu_2\in {\cal B}_D(\mu_1)$. By Proposition \ref{multip}, there exists an isometry $V\in \C^{m\times 2}$ having the form $V=\begin{bmatrix} v_{11}& v_{12}\\ 0& v_{22}\end{bmatrix}$ where $v_{11},v_{12}\in \C^n$ such that $\begin{bmatrix} \mu_1& *\\ 0& \mu_2\end{bmatrix}=V^*D V$.
Hence, \[\mu_2=v_{12}^*\Lambda v_{12}+v_{22}^*\Lambda'v_{22}.\] If $v_{12}=0$, then the assertion holds. Assume both $v_{12}$ and $v_{22}$ are nonzero. Then \[\mu_2=||v_{12}||^2\omega+||v_{22}||^2\zeta\] is a convex combination of $\omega=\dfrac{v_{12}^*\Lambda v_{12}}{v_{12}^*v_{12}}$ and $\zeta=\dfrac{v_{22}^*\Lambda' v_{22}}{v_{22}^*v_{22}}$ since $||v_{12}||^2+||v_{22}||^2=1$. Clearly, $\zeta \in W(\Lambda')$. To see why $\omega\in {\cal B}_\Lambda(\mu_1)$, observe that $\begin{bmatrix} \mu_1& *\\ 0& \omega\end{bmatrix}=W^*\Lambda W$ where $W=\begin{bmatrix}v_{11}& \dfrac{v_{12}}{||v_{12}||}\end{bmatrix}\in \C^{n\times 2}$ is an isometry. Finally, if $v_{12}\neq 0$ but $v_{22}=0$, then the same $W$ works.

For the second part, assume $\textup{conv}[{\cal B}_\Lambda(\mu_1)]={\cal R}_\Lambda(\mu_1)$. Observe that the generating points of ${\cal R}_D(\mu_1)$ are all in ${\cal B}_D(\mu_1)$, and so it suffices to verify that ${\cal B}_D(\mu_1)\subseteq {\cal R}_D(\mu_1)$. The first part implies
\[{\cal B}_D(\mu_1)=\textup{conv}[{\cal B}_\Lambda(\mu_1)\cup W(\Lambda')]=\textup{conv}[{\cal R}_\Lambda(\mu_1)\cup W(\Lambda')].\]
Direct computation reveals that $\textup{conv}[{\cal R}_\Lambda(\mu_1)\cup W(\Lambda')]\subseteq {\cal R}_D(\mu_1)$. 
\end{proof}

Given distinct $\lambda_1,\ldots,\lambda_n\in\C$ with counterclockwise orientation with respect to $\displaystyle\sum_{j=1}^n\lambda_j/n$ such that no $\lambda_j$ is in the interior of $\textup{conv}\{\lambda_j:j=1,\ldots n\}$, let $\mu_1\in \textup{conv}\{\lambda_a,\lambda_b,\lambda_c\}$ be given where $1\leq a<b<c\leq n$. When $\lambda_a,\lambda_b,\lambda_c$ do not lie on the same line, let $r_i^{(abc)}$'s be the unique convex weights of $\mu_1$ from $
\mu_1=r_a^{(abc)}\lambda_a+r_b^{(abc)}\lambda_b+r_c^{(abc)}\lambda_c
$ where $r_i^{(abc)}\geq0$ and $\displaystyle\sum_i r_i^{(abc)}=1$. We can also characterize the convex weights as follows:
\begin{equation}\label{area} 
r_i^{(abc)}=\dfrac{\textup{Area of}\ \textup{conv}[\{\mu_1,\lambda_a,\lambda_b,\lambda_c\}\setminus\{\lambda_i\}]}{\textup{Area of}\ \textup{conv}\{\lambda_a,\lambda_b,\lambda_c\}}
\end{equation}
for each $i=a,b,c$. Define
\[r^{(abc)}:=r_a^{(abc)}e_a+r_b^{(abc)}e_b+r_c^{(abc)}e_c.\]

If $\mu_1\in (\lambda_a,\lambda_b)$ where $1\leq a<b\leq n$, let $r_i^{(ab)}$ be the unique convex weights from $\mu_1=r_a^{(ab)}\lambda_a+r_b^{(ab)}\lambda_b$ where $r_i^{(ab)}\geq 0$ and $\displaystyle\sum_{i}r_i^{(ab)}=1$. In this case, we can characterize the convex weights as follows:
\begin{equation}\label{length}
r_i^{(ab)}=\dfrac{\textup{Length of}\ \textup{conv}[\{\mu_1,\lambda_a,\lambda_b\}\setminus\{\lambda_i\}]}{\textup{Length of}\ \textup{conv}\{\lambda_a,\lambda_b\}}
\end{equation} 
Define 
\[r^{(ab)}:=r_a^{(ab)}\lambda_a+r_b^{(ab)}\lambda_b.\] 

\begin{lemma}\label{oneone}
Given $n\geq 4$, let $\Lambda=\textup{diag}(\lambda_1,\ldots,\lambda_n)$ with distinct $\lambda_j$'s not lying on the same line and arranged in a counterclockwise orientation with respect to $\textup{trace}(\Lambda)/n$ such that no eigenvalue is in the interior of $W(\Lambda)$. Assume $\lambda_2$ is a corner, and let $\mu_1\in \textup{conv}\{\lambda_1,\lambda_2,\lambda_3\}$ be an interior point. Then $r_1^{(123)}t_2-t_1r_2^{(123)}>0$ for any $t\in {\cal C}_\Lambda(\mu_1)\setminus\{r^{(123)}\}$.
\end{lemma}
\begin{proof}
Observe that
\[\begin{array}{rcl}
\mu_1-\lambda_3&=&r_1^{(12j)}(\lambda_1-\lambda_3)+r_2^{(12j)}(\lambda_2-\lambda_3)+r_j^{(12j)}(\lambda_j-\lambda_3)\\
&=&\displaystyle\sum_{j=1}^nt_j(\lambda_j-\lambda_3).\end{array}\] Thus, we can assume that $\lambda_3=0$. 
It suffices to verify
\[r_1^{(123)}r_2^{(abc)}-r_1^{(abc)}r_2^{(123)}>0\ \textup{and}\ r_1^{(123)}r_2^{(ab)}-r_1^{(ab)}r_2^{(123)}>0\]
for all extreme points $r^{(abc)},r^{(ab)}\in\textup{Ext}[{\cal C}_\Lambda(\mu_1)]\setminus\{r^{(123)}\}$. Since $\mu_1$ is an interior point of $\textup{conv}\{\lambda_1,\lambda_2,\lambda_3\}$ and $\lambda_2$ is between $\lambda_1$ and $\lambda_3$, note that $r_1^{(ab)}=0$ and $r_2^{(ab)}>0$, and thus the assertion holds in this case. Similarly, $r_2^{(abc)}>0$. If $r_1^{(abc)}=0$, then the assertion holds. If $r_1^{(abc)}\neq0$, then we can take $\lambda_a=\lambda_1$ and $\lambda_b=\lambda_2$. Note that
\[r_1^{(123)}r_2^{(12c)}-r_1^{(12c)}r_2^{(123)}=\dfrac{\textup{Im}[\overline{(r_3^{(12c)}\lambda_c-\mu_1)}(0-\mu_1)]}{\textup{Im}[\overline{(\lambda_2-0)}(\lambda_1-0)]}>0.\]
\end{proof}

For $z_1,z_2\in \C$ distinct points, $\textup{RHS}[z_1,z_2]$ denotes 
the strict right hand side of the line along $[z_1,z_2]$ with direction $z_2-z_1$. The notation $\overline{\textup{RHS}[z_1,z_2]}$ is the closure of $\textup{RHS}[z_1,z_2]$, and $\textup{LHS}[z_1,z_2]$ and $\overline{\textup{LHS}[z_1,z_2]}$ are defined analogously.

\section{$\mu_1\notin \Lambda_2(A)$: ${\cal B}_A(\mu_1)\subseteq \overline{\textup{RHS}[\lambda_3,w_{123}]}$}

By applying a rotation or translation, we assume the following:
\begin{enumerate}
\item[(A1)] Given $n\geq 4$, let $A\in\C^{n\times n}$ be normal with distinct eigenvalues $\lambda_1,\ldots,\lambda_n$ not lying on the same line and arranged in a counterclockwise orientation with respect to $\textup{trace}(A)/n$ such that no eigenvalue is in the interior of $W(A)$.
\item[(A2)] $\textup{Re}(\lambda_2)<0$, $\lambda_3=0$, and $\textup{Re}(\lambda_j)> 0$ for all $\lambda_j\neq \lambda_2,\lambda_3$.   
\item[(A3)] $\mu_1\notin \partial W(A)$, $\mu_1\notin \Lambda_2(A)$, and $\mu_1\in \textup{conv}\{\lambda_1,\lambda_2,\lambda_3\}$ is an interior point.
\item[(A4)] $\textup{Re}(w_{123})=0$ where $w_{123} $ is the isogonal conjugate of $\mu_1$ with respect to $\textup{conv}\{\lambda_1,\lambda_2,\lambda_3\}$.
\end{enumerate}

\begin{lemma}\label{sine}
Let $A\in\C^{n\times n}$ be normal satisfying (A1)-(A2) and $\mu_1\in W(A)$ satisfying (A3)-(A4). Then
\[\sin[2\arg(\lambda_2)-\arg(\mu_1)]>\sin[2\arg(\lambda_j)-\arg(\mu_1)]\]
for all $\lambda_j\neq\lambda_1,\lambda_2,\lambda_3$.
\end{lemma}
\begin{proof}
Let $\arg(\lambda_2)\in \left[\frac{\pi}{2},\frac{3\pi}{2}\right]$ and $\arg(\lambda_j)\in \left[-\frac{\pi}{2},\frac{\pi}{2}\right]$ for all $\lambda_j\neq\lambda_2,\lambda_3$. By (A1)-(A2), the arguments can be chosen so that $\arg(\lambda_j)\in[\arg(\lambda_2)-\pi, \arg(\lambda_1)]$ for all $\lambda_j\neq \lambda_2,\lambda_3$. If $\textup{Re}(\mu_1)>0$, take $\arg(\mu_1)\in \left[-\frac{\pi}{2},\frac{\pi}{2}\right]$ and if $\textup{Re}(\mu_1)\leq0$, take $\arg(\mu_1)\in \left[\frac{\pi}{2},\textup{arg}(\lambda_2)\right].$

Write $\arg(\lambda_2)=\theta+\varphi$ where $\theta>0$ and
\[\varphi=\begin{cases}\textup{arg}(\mu_1),& \textup{if}\ \textup{Re}(\mu_1)\leq 0\\
\frac{\pi}{2},& \textup{if}\ \textup{Re}(\mu_1)>0.
\end{cases}\]
Define
\[u(\theta)=\begin{cases}\frac{\pi}{2}-\theta,& \textup{if}\ \textup{Re}(\mu_1)\leq 0\\
\arg(\mu_1)-\theta,& \textup{if}\ \textup{Re}(\mu_1)>0.
\end{cases}\]
Due to assumption (A4), $[\lambda_3,\mu_1]$ and $[\lambda_3,w_{123}]$ are symmetric about the angle bisector of the vertex angle at $\lambda_3$. As a consequence, $u(\theta)=\arg(\lambda_1)$. 

Let $\lambda_j\neq\lambda_1,\lambda_2,\lambda_3$. Note that 
\[\textup{arg}(\lambda_2)-\pi\leq \arg(\lambda_j)\leq \arg(\lambda_1)=u(\theta)\]
and \[2\arg(\lambda_2)-\arg(\mu_1)\geq\arg(\lambda_2)+\theta>\arg(\lambda_2).\]
Hence,
\small
\[2\arg(\lambda_2)-\arg(\mu_1)\leq2\pi+2\arg(\lambda_j)-\arg(\mu_1)\leq 2\pi+ 2u(\theta)-\arg(\mu_1).\]
\normalsize
The upper bound can be simplified as
\[ 2\pi+2u(\theta)-\arg(\mu_1)=3\pi-(2\arg(\lambda_2)-\arg(\mu_1)).\]

If we set $\alpha=2\arg(\lambda_2)-\arg(\mu_1)$ and $x=2\pi+2\arg(\lambda_j)-\arg(\mu_1)$, then $x\in [\alpha,3\pi-\alpha]$. Since $0<\arg(\lambda_2)<\alpha$ and the endpoints of $[\alpha,3\pi-\alpha]$ are symmetric about $\frac{3\pi}{2}$, we have that $x\in [\alpha,3\pi-\alpha]\subseteq \left[\frac{\pi}{2},\frac{5\pi}{2}\right]$.

Note that $\sin(t)$ is strictly decreasing on $\left[\alpha,\frac{3\pi}{2}\right]$ and strictly increasing on $\left[\frac{3\pi}{2},3\pi-\alpha\right]$. If $x\in\left[\alpha,\frac{3\pi}{2}\right]$, then 
\[\sin[2\arg(\lambda_2)-\arg(\mu_1)]=\sin(\alpha)>\sin(x)=\sin[2\arg(\lambda_j)-\arg(\mu_1)].\]
If $x\in \left[\frac{3\pi}{2},3\pi-\alpha\right]$, then
\[\sin[2\arg(\lambda_j)-\arg(\mu_1)]=\sin(x)<\sin(3\pi-\alpha)=\sin(\alpha)=\sin[2\arg(\lambda_2)-\arg(\mu_1)].\]
\end{proof}

\begin{lemma}\label{positive}
Let $A\in\C^{n\times n}$ be normal satisfying (A1)-(A2) and $\mu_1\in W(A)$ satisfying (A3)-(A4). If $\lambda_j\neq\lambda_1,\lambda_2,\lambda_3$, then \[\dfrac{\textup{Im}(\mu_1\overline{\lambda_j})}{\textup{Im}(\overline{\mu_1}\lambda_2)}+\dfrac{\textup{Re}(\lambda_2)}{|\lambda_2|^2}\cdot\dfrac{|\lambda_j|^2}{\textup{Re}(\lambda_j)}>0\] if and only if 
\[\sin[2\arg(\lambda_j)-\arg(\mu_1)]<\sin[2\arg(\lambda_2)-\arg(\mu_1)].\]
\end{lemma}
\begin{proof}
Note that $\textup{Re}(\lambda_j)$ and $\textup{Im}(\overline{\mu_1}\lambda_2)$ are both positive, and so we can rewrite the assertion as follows:
\[\dfrac{\textup{Re}(\lambda_j)}{|\lambda_j|^2}\textup{Im}(\mu_1\overline{\lambda_j})>-\dfrac{\textup{Re}(\lambda_2)}{|\lambda_2|^2}\textup{Im}(\overline{\mu_1}\lambda_2).\]
Dividing both sides by $|\mu_1|$, we get the equivalent assertion
\[\begin{array}{rcl}\cos[\arg(\lambda_j)]\sin[\arg(\mu_1)-\arg(\lambda_j)]&=&\dfrac{\textup{Re}(\lambda_j)}{|\lambda_j|}\textup{Im}\left(\dfrac{\mu_1}{|\mu_1|}\cdot\dfrac{\overline{\lambda_j}}{|\lambda_j|}\right)\\
&>&\dfrac{\textup{Re}(\lambda_2)}{|\lambda_2|}\textup{Im}\left(\dfrac{\mu_1}{|\mu_1|}\cdot\dfrac{\overline{\lambda_2}}{|\lambda_2|}\right)\\
&=&\cos[\arg(\lambda_2)]\sin[\arg(\mu_1)-\arg(\lambda_2)].\end{array}
\]
By using a product-to-sum identity, this last inequality is equivalent to 
\[\sin[2\arg(\lambda_j)-\arg(\mu_1)]<\sin[2\arg(\lambda_2)-\arg(\mu_1)].\]
\end{proof}

\begin{lemma}\label{rdiff}
Let $A\in\C^{n\times n}$ be normal satisfying (A1)-(A2) and $\mu_1\in W(A)$ satisfying (A3)-(A4). The following statements hold:
\begin{enumerate}
\item[(i)] $\textup{Re}(\lambda_1)=-\dfrac{r_1^{(123)}}{r_2^{(123)}}\cdot\dfrac{|\lambda_1|^2}{|\lambda_2|^2}\textup{Re}(\lambda_2)$ and $\textup{Im}(\overline{\mu_1}\lambda_2)=r_1^{(123)}\textup{Im}(\overline{\lambda_1}\lambda_2).$

\item[(ii)] Suppose $\mu_1\in \textup{conv}\{\lambda_1,\lambda_2,\lambda_j\}\cup\textup{conv}\{\lambda_2,\lambda_3,\lambda_j\}$ for some $\lambda_j\neq\lambda_1,\lambda_2,\lambda_3$. Then $\dfrac{\textup{Im}(\overline{\lambda_j}\mu_1)}{\textup{Im}(\overline{\lambda_1}\lambda_2)}=\dfrac{r_1^{(123)}r_2^{(y)}-r_1^{(y)}r_2^{(123)}}{r_j^{(y)}}$ where $(y)=(12j),(23j),$ or $(2j)$.

\end{enumerate}
\end{lemma}
\begin{proof} By Proposition \ref{isogformula1} and assumptions (A2) and (A4), there exists $c>0$ such that
\[0=c\textup{Re}(w_{123})=\dfrac{|\lambda_2|^2}{r_1^{(123)}}\textup{Re}(\lambda_1)+\dfrac{|\lambda_1|^2}{r_2^{(123)}}\textup{Re}(\lambda_2),\] and hence the first part of assertion (i) holds. Moreover, note that 
\[\begin{array}{rcl}r_1^{(123)}\textup{Im}(\overline{\lambda_1}\lambda_2)&=&2r_1^{(123)}\left[\frac{1}{2}\textup{Im}(\overline{\lambda_1}\lambda_2+\overline{\lambda_2}\lambda_3+\overline{\lambda_3}\lambda_1)\right]\\
&=& 2r_1^{(123)}[\textup{Area of}\ \textup{conv}\{\lambda_1,\lambda_2,\lambda_3\}]\\
&=&2[\textup{Area of}\ \textup{conv}\{\mu_1,\lambda_2,\lambda_3\}]\\
&=&\textup{Im}(\overline{\mu_1}\lambda_2)\end{array}\]
due to \eqref{area} and \cite[Exercise 10 on p. 40]{rH91}. 

Suppose $\mu_1\in\textup{conv}\{\lambda_1,\lambda_2,\lambda_j\}$ for some $\lambda_j\neq\lambda_1,\lambda_2,\lambda_3$. For $(y)=(12j),(23j),(2j)$, $\mu_1=r_1^{(y)}\lambda_1+r_2^{(y)}\lambda_2+r_3^{(y)}\lambda_3+r_j^{(y)}\lambda_j$ where $\lambda_3=0$. On one hand,
\[\textup{Im}[\overline{(r_j^{(y)}\lambda_j-\mu_1)}(0-\mu_1)]=-\textup{Im}[r_j^{(y)}\overline{\lambda_j}\mu_1]=r_j^{(y)}\textup{Im}(\lambda_j\overline{\mu_1}).\]
On the other hand,
\[\begin{array}{rcl}
\textup{Im}[\overline{(r_j^{(y)}\lambda_j-\mu_1)}(0-\mu_1)]&=&\textup{Im}[(-r_1^{(y)}\overline{\lambda_1}-r_2^{(y)}\overline{\lambda_2})(-r_1^{(123)}\lambda_1-r_2^{(123)}\lambda_2)]\\
&=&(r_1^{(123)}r_2^{(y)}-r_1^{(y)}r_2^{(123)})\textup{Im}(\overline{\lambda_2}\lambda_1).
\end{array}\] Observe that $\textup{Im}(\overline{\lambda_2}\lambda_1)\neq 0$ since $\lambda_1,\lambda_2,$ and $\lambda_3=0$ do not lie on the same line, and so assertion (ii) follows.\end{proof}

For $z_1,z_2,z_3\in\C$ that do not lie on the same line, $\angle (z_1,z_2,z_3)$ denotes the acute angle whose vertex is at $z_2$ and has terminal sides at $z_1$ and $z_3$. 

\begin{lemma}\label{xt}
Let $A\in\C^{n\times n}$ be normal satisfying (A1)-(A2) and $\mu_1\in W(A)$ satisfying (A3)-(A4). If $x:=\displaystyle\sum_{j\neq 3}\dfrac{|\lambda_j|^2}{\textup{Re}(\lambda_j)}e_j$ where $e_j$ is the $j^{th}$ standard basis vector in $\C^n$, then $x^Tt<0$ for any $t\in {\cal C}_\Lambda(\mu_1)\setminus\{r^{(123)}\}$.
\end{lemma}
\begin{proof}
It suffices to verify
\[\label{ineqext}x^Tr^{(y)}<0\]
for all extreme points $r^{(y)}\in \textup{Ext}[{\cal C}_\Lambda(\mu_1)]\setminus\{r^{(123)}\}$. Similar to the proof of Lemma \ref{oneone}, we may assume $\lambda_2\in \{\lambda_a,\lambda_b,\lambda_c\}$ when $(y)=(abc)$ or $\lambda_2\in\{\lambda_a,\lambda_b\}$ when $(y)=(ab)$.\\

\noindent\textbf{Case 1}: $(y)=(12c)$ where $3<c\leq n$.\\

By Lemma \ref{rdiff},
\[\begin{array}{rcl}
\vspace{0.2cm}
x^Tr^{(12c)}&=&\dfrac{r_1^{(12c)}}{\textup{Re}(\lambda_1)}|\lambda_1|^2+
\dfrac{r_2^{(12c)}}{\textup{Re}(\lambda_2)}|\lambda_2|^2+\dfrac{r_c^{(12c)}}{\textup{Re}(\lambda_c)}|\lambda_c|^2\\
\vspace{0.2cm}
&=&\dfrac{|\lambda_2|^2}{r_1^{(123)}\textup{Re}(\lambda_2)}[r_1^{(123)}r_2^{(12c)}-r_2^{(123)}r_1^{(12c)}]+\dfrac{r_c^{(12c)}}{\textup{Re}(\lambda_c)}|\lambda_c|^2\\
\vspace{0.2cm}
&=&\dfrac{|\lambda_2|^2}{r_1^{(123)}\textup{Re}(\lambda_2)}\cdot\dfrac{r_c^{(12c)}\textup{Im}(\overline{\lambda_c}\mu_1)}{\textup{Im}(\overline{\lambda_1}\lambda_2)}+\dfrac{r_c^{(12c)}}{\textup{Re}(\lambda_c)}|\lambda_c|^2\\
\vspace{0.2cm}
&=&\dfrac{r_c^{(12c)}|\lambda_2|^2\textup{Im}(\mu_1\overline{\lambda_c})}{[r_1^{(123)}\textup{Im}(\overline{\lambda_1}\lambda_2)]\textup{Re}(\lambda_2)}+\dfrac{r_c^{(12c)}}{\textup{Re}(\lambda_c)}|\lambda_c|^2\\
\vspace{0.2cm}
&=&\dfrac{r_c^{(12c)}|\lambda_2|^2\textup{Im}(\mu_1\overline{\lambda_c})}{\textup{Im}(\overline{\mu_1}\lambda_2)\textup{Re}(\lambda_2)}+\dfrac{r_c^{(12c)}}{\textup{Re}(\lambda_c)}|\lambda_c|^2\\

\vspace{0.2cm}
&=&\dfrac{r_c^{(12c)}|\lambda_2|^2}{\textup{Re}(\lambda_2)}\left[\dfrac{\textup{Im}(\mu_1\overline{\lambda_c})}{\textup{Im}(\overline{\mu_1}\lambda_2)}+\dfrac{\textup{Re}(\lambda_2)}{|\lambda_2|^2}\dfrac{|\lambda_c|^2}{\textup{Re}(\lambda_c)}\right].
\end{array}\]
Thus, $x^Tr^{(12c)}<0$ holds due to Lemmas \ref{sine} and \ref{positive}.

\noindent\textbf{Case 2}: $(y)=(23c)$ where $3< c\leq n$.\\

Using \eqref{area} and \cite[Exercise 10 on p. 40]{rH91}, $x^Tr^{(23c)}<0$ if and only if  
\[\dfrac{|\lambda_c|^2}{\textup{Re}(\lambda_c)}\dfrac{\textup{Im}(\overline{\mu_1}\lambda_2)}{\textup{Im}(\overline{\lambda_c}\lambda_2)} +\dfrac{|\lambda_2|^2}{\textup{Re}(\lambda_2)}\dfrac{\textup{Im}(\overline{\lambda_c}\mu_1)}{\textup{Im}(\overline{\lambda_c}\lambda_2)}<0.\]
Equivalently,
\[\dfrac{\textup{Im}(\mu_1\overline{\lambda_c})}{\textup{Im}(\overline{\mu_1}\lambda_2)}+\dfrac{\textup{Re}(\lambda_2)}{|\lambda_2|^2}\dfrac{|\lambda_c|^2}{\textup{Re}(\lambda_c)}>0\]
since $\textup{Re}(\lambda_2)<0$ and $\textup{Im}(\overline{\mu_1}\lambda_2)>0$. Thus, $x^Tr^{(23c)}<0$ holds due to Lemmas \ref{sine} and \ref{positive}.

\noindent\textbf{Case 3}: $(y)=(2bc)$ where $3<b< c\leq n$ or $(y)=(2b)$ where $3<b\leq n$.\\

Let $\lambda_b'$ be the intersection of $[\lambda_2,\lambda_b]$ and the line containing $[\mu_1,\lambda_3]$. Similarly, let $\lambda_c'$ be the intersection of $[\lambda_2,\lambda_c]$ and the line containing $[\mu_1,\lambda_3]$. If we let $\mu=x+iy$ vary in $[\lambda_b',\lambda_c']$, we can view the statement $x^Tr^{(y)}<0$ as an optimization problem of the real affine function $x^Tr^{(y)}$ over $[\lambda_b',\lambda_c']$. Thus, it suffices to check $x^Tr^{(y)}<0$ at the endpoints $\lambda_b'$ and $\lambda_c'$. Note that $\mu_1\in [\lambda_b',\lambda_c']$, and so $\arg(\mu_1)=\arg(\mu)$. We only verify $x^Tr^{(y)}<0$ at $\mu=\lambda_b'$ as the case when $\mu=\lambda_c'$ is analogous. 

By \eqref{length}, 
\[\mu=\lambda_b'=r_2^{(2b)}\lambda_2+r_b^{(2b)}\lambda_b=\dfrac{|\lambda_b'-\lambda_b|}{|\lambda_2-\lambda_b|}\lambda_2+\dfrac{|\lambda_b'-\lambda_2|}{|\lambda_2-\lambda_b|}\lambda_b,\]
and so \[|\lambda_2-\lambda_b|x^Tr^{(y)}=\dfrac{|\lambda_b'-\lambda_b|}{\textup{Re}(\lambda_2)}|\lambda_2|^2+\dfrac{|\lambda_b'-\lambda_2|}{\textup{Re}(\lambda_b)}|\lambda_b|^2.\]
Hence, $x^Tr^{(y)}<0$ if and only if 
\[\begin{array}{rcl}\cos[\arg(\lambda_b)]\dfrac{|\lambda_b'-\lambda_b|}{|\lambda_b|}&=&\dfrac{\textup{Re}(\lambda_b)|\lambda_b'-\lambda_b|}{|\lambda_b|^2}\\
&>&-\dfrac{\textup{Re}(\lambda_2)|\lambda_b'-\lambda_2|}{|\lambda_2|^2}\\
&=&-\cos[\arg(\lambda_2)]\dfrac{|\lambda_b'-\lambda_2|}{|\lambda_2|}.
\end{array}\] We rewrite the expressions $\dfrac{|\lambda_b'-\lambda_b|}{|\lambda_b|}$ and $\dfrac{|\lambda_b'-\lambda_2|}{|\lambda_2|}$ in terms of the sine function.

By considering $\textup{conv}\{\mu,\lambda_2,\lambda_3\}$, Law of Sines guarantees
\[\dfrac{|\lambda_b'-\lambda_2|}{\sin[\angle(\lambda_2,\lambda_3,\mu)]}=\dfrac{|\lambda_2-\lambda_3|}{\sin[\angle(\lambda_3,\mu,\lambda_2)]}=\dfrac{|\lambda_2|}{\sin[\angle(\lambda_3,\mu,\lambda_2)]},\]
and hence
\[\dfrac{|\lambda_b'-\lambda_2|}{|\lambda_2|}=\dfrac{\sin[\angle(\lambda_2,\lambda_3,\mu)]}{\sin[\angle(\lambda_3,\mu,\lambda_2)]}=\dfrac{\sin[\arg(\lambda_2)-\arg(\mu_1)]}{\sin[\angle(\lambda_3,\mu,\lambda_2)]}\]
since $\arg(\lambda_2)=\arg(\mu_1)+\angle(\lambda_2,\lambda_3,\mu)$. Similarly, by applying Law of Sines on $\textup{conv}\{\mu,\lambda_3,\lambda_b\}$, we get
\[\dfrac{|\lambda_b'-\lambda_b|}{\sin[\angle(\lambda_2,\lambda_3,\lambda_b)-\angle(\lambda_2,\lambda_3,\mu)]}=\dfrac{|\lambda_b-\lambda_3|}{\sin[\pi-\angle(\lambda_3,\mu,\lambda_2)]}=\dfrac{|\lambda_b|}{\sin[\angle(\lambda_3,\mu,\lambda_2)]},\]
and hence
\[\dfrac{|\lambda_b'-\lambda_b|}{|\lambda_b|}=\dfrac{\sin[\angle(\lambda_2,\lambda_3,\lambda_b)-\angle(\lambda_2,\lambda_3,\mu)]}{\sin[\angle(\lambda_3,\mu,\lambda_2)]}=\dfrac{\sin[\arg(\mu_1)-\arg(\lambda_b)]}{\sin[\angle(\lambda_3,\mu,\lambda_2)]}.\]
 Thus, $x^Tr^{(y)}<0$ if and only if 
\[\cos[\arg(\lambda_b)]\sin[\arg(\mu_1)-\arg(\lambda_b)]>-\cos[\arg(\lambda_2)]\sin[\arg(\lambda_2)-\arg(\mu_1)].\]
By using a product-to-sum identity, this last inequality is equivalent to \[\sin[2\arg(\lambda_b)-\arg(\mu_1)]<\sin[2\arg(\lambda_2)-\arg(\mu_1)].\] Thus, $x^Tr^{(y)}<0$ holds due to Lemma \ref{sine}. 
\end{proof}

\begin{proposition}\label{realpartmup1}
Let $A\in\C^{n\times n}$ be normal satisfying (A1)-(A2) and $\mu_1\in W(A)$ satisfying (A3)-(A4). If $t\in {\cal C}_\Lambda(\mu_1)$ has all positive entries, then \[{\cal B}_A(\mu_1,t)\subseteq \textup{conv}[\{w_{123}\}\cup(\sigma(A)\ominus\{\lambda_2\})].\]
\end{proposition}
\begin{proof}
We only prove that ${\cal B}_A(\mu_1,t)\subseteq \overline{\textup{RHS}[\lambda_3,w_{123}]}$. The proof for the inclusion ${\cal B}_A(\mu_1,t)\subseteq \overline{\textup{LHS}[\lambda_1,w_{123}]}$ is similar. 

Elements of ${\cal B}_A(\mu_1,t)$ are of the form $\frac{v^*\Lambda v}{v^*v}$ for all nonzero $v\in \{\sqrt{t},\Lambda\sqrt{t}\}^\perp$. By (A1)-(A4), it suffices to prove $\textup{Re}(v^*\Lambda v)\geq 0$. A basis for $\{\sqrt{t},\Lambda\sqrt{t}\}^\perp$ is given by
\[f_j=\dfrac{\overline{\lambda_j-\lambda_2}}{\sqrt{t_1}}e_1+\dfrac{\overline{\lambda_1-\lambda_j}}{\sqrt{t_2}}e_2+\dfrac{\overline{\lambda_2-\lambda_1}}{\sqrt{t_j}}e_j\]
for $j=3,\ldots,n$ where $e_k$ is the $k^{th}$ standard basis vector in $\C^n$. If $w=[\lambda_3\ \cdots\ \lambda_n]^T$, $\mathbbm 1=[1\ \cdots\ 1]^T\in \C^{n-2}$, and $Z(t)$ be as defined in \eqref{psd}, then $\textup{Re}(v^*\Lambda v)\geq0$ for all $v\in \{\sqrt{t},\Lambda\sqrt{t}\}^\perp$ if and only if
$Z(t)\geq 0$. Set $x=[-\lambda_2\ 1]^T$ and $y=[\lambda_1\ -1]^T$,  
\[
X=\dfrac{\textup{Re}(\lambda_1)}{t_1}xx^*+\dfrac{\textup{Re}(\lambda_2)}{t_2}yy^*,\]
\[Y=[\mathbbm 1\ w],\]
\[
\Delta=|\lambda_2-\lambda_1|^2\textup{diag}\left(\dfrac{\textup{Re}(\lambda_3)}{t_3},\ldots,\dfrac{\textup{Re}(\lambda_n)}{t_n}\right),\]
and
\[\Delta_\epsilon=\Delta+\epsilon|\lambda_2-\lambda_1|^2e_1e_1^*\]
where $\epsilon>0$ and $e_j$ is the $j^{th}$ standard basis vector in $\C^{n-2}$. Then $Z(t)\geq 0$ is equivalent to 
\begin{equation}\label{ver2}
YXY^*+\Delta\geq0.
\end{equation}

We prove \eqref{ver2} by showing $YXY^*+\Delta_\epsilon>0$ and then letting $\epsilon\to 0$. Note that $\Delta_\epsilon>0$ from assumption (A2). Let $\Delta_\epsilon^{-\frac{1}{2}}Y=Q_\epsilon R_\epsilon$ be a $QR$ factorization. The matrix 
\begin{equation}\label{ver3}
YXY^*+\Delta_\epsilon>0
\end{equation}
 if and only if 
\begin{equation}\label{ver4}
R_\epsilon XR_\epsilon^*+I_2>0.
\end{equation}
 Equivalently, 
\begin{equation}\label{ver4a}
\textup{trace}(R_\epsilon X R_\epsilon^*)+2=\textup{trace}(R_\epsilon XR_\epsilon^*+I_2)>0
\end{equation}
and
\begin{equation}\label{ver4b}
1+\textup{trace}(R_\epsilon X R_\epsilon^*)+\textup{det}(R_\epsilon X R_\epsilon^*)=\textup{det}(R_\epsilon X R_\epsilon^*+I_2)>0
\end{equation}
If we can show that the coefficients of $1/\epsilon$ in \eqref{ver4a} and \eqref{ver4b} are positive, then \eqref{ver2} follows.

For some $b(t)\in\R$, direct computations reveal that
\[
\begin{array}{rcl}
\textup{trace}(R_\epsilon XR_\epsilon^*)&=&\textup{trace}(X Y^*\Delta_\epsilon^{-1} Y)\\
&=&\dfrac{1}{|\lambda_2-\lambda_1|^2} a(t)\dfrac{1}{\epsilon}+b(t).
\end{array}\]
where 
\[\begin{array}{rcl}
\vspace{0.2cm}
a(t)&=&\dfrac{\textup{Re}(\lambda_1)}{t_1}|\lambda_2|^2+\dfrac{\textup{Re}(\lambda_2)}{t_2}|\lambda_1|^2\\
\vspace{0.2cm}
&=&-\dfrac{|\lambda_1|^2\textup{Re}(\lambda_2)}{r_2^{(123)}t_1t_2}[r_1^{(123)}t_2-t_1r_2^{(123)}]
\end{array}\]
due to Lemma \ref{rdiff}. Since $n\geq 4$ and $t$ has all positive entries, $t\neq r^{(123)}$. Lemma \ref{oneone} guarantees $r_1^{(123)}t_2-t_1r_2^{(123)}>0$, and hence $a(t)>0$ due to assumption (A2).

Similarly, there exists $d(t)\in \R$ such that
\[\begin{array}{rcl}
\textup{det}(R_\epsilon X R_\epsilon^*+I_2)
&=&\textup{det}(XY^*\Delta_\epsilon^{-1} Y+I_2)\\
&=&\dfrac{1}{|\lambda_2-\lambda_1|^2}c(t)\dfrac{1}{\epsilon}+d(t)
\end{array}\]
where 
\[
\begin{array}{rcl}c(t)&=&
\dfrac{\textup{Re}(\lambda_1)}{t_1}|\lambda_2|^2+\dfrac{\textup{Re}(\lambda_2)}{t_2}|\lambda_1|^2+\dfrac{\textup{Re}(\lambda_1)\textup{Re}(\lambda_2)}{t_1t_2} \left[\displaystyle\sum_{j=4}^n\dfrac{t_j}{\textup{Re}(\lambda_j)}|\lambda_j|^2\right]\\
&=&\dfrac{\textup{Re}(\lambda_1)\textup{Re}(\lambda_2)}{t_1t_2}\left[\displaystyle\sum_{j\neq 3}^n\dfrac{t_j}{\textup{Re}(\lambda_j)}|\lambda_j|^2\right].
\end{array}\]

Since $\textup{Re}(\lambda_1)>0$ and $\textup{Re}(\lambda_2)<0$, the quantity $c(t)>0$ if and only if 
\[\displaystyle\sum_{j\neq 3}^n\dfrac{t_j}{\textup{Re}(\lambda_j)}|\lambda_j|^2<0.\] 
This last inequality holds due to Lemma \ref{xt}.\end{proof}

\begin{lemma}\label{keylemma2}
Let $x_1,\ldots,x_m>0$ and $z_1,\ldots,z_m\in \C$. Then \[\dfrac{\left|\displaystyle\sum_{j=1}^mz_j\right|^2}{\displaystyle\sum_{j=1}^mx_j}\leq\displaystyle\sum_{j=1}^m\dfrac{|z_j|^2}{x_j}.\]
\end{lemma}
\begin{proof}
Let $x=[\sqrt{x_1}\ \cdots\ \sqrt{x_m}]^T\in \mathbb C^m$ and $y=[z_1/\sqrt{x_1}\ \cdots\ z_m/\sqrt{x_m}]^T\in \mathbb C^m$. By the Cauchy-Schwarz inequality,
\[\left|\displaystyle\sum_{j=1}^mz_j\right|^2=|\langle x,y\rangle|^2\leq \left( \displaystyle\sum_{j=1}^mx_j\right)\left(\displaystyle\sum_{j=1}^m\dfrac{|z_j|^2}{x_j}\right).\]
\end{proof}

\begin{lemma}\label{dec28l1}
Let $A\in\C^{5\times 5}$ be normal satisfying (A1)-(A2) and $\mu_1\in W(A)$ satisfying (A3)-(A4). Assume that $\lambda_4$ is on the line through $[\lambda_2,\lambda_3]$ (possibly, $\lambda_4=\lambda_3$) and $\lambda_5$ is on the line through $[\lambda_2,\lambda_1]$ (possibly, $\lambda_5=\lambda_1$). If $\mu_1\in \textup{conv}\{\lambda_2,\lambda_4,\lambda_5\}$, then $w_{245}\in \textup{conv}[\{w_{123}\}\cup (\sigma(A)\ominus\{\lambda_2\})]$.
\end{lemma}
\begin{proof}
We only prove that $w_{245}\in \overline{\textup{RHS}[\lambda_3,w_{123}]}$. The proof for the inclusion $w_{245}\in \overline{\textup{LHS}[\lambda_1,w_{123}]}$ is similar.

If $\mu_1\in (\lambda_2,\lambda_4)\cup(\lambda_2,\lambda_5)$, then $w_{245}=\lambda_5$ or $w_{245}=\lambda_4$, and hence the assertion holds in either case. Assume $\mu_1\in \textup{conv}\{\lambda_2,\lambda_4,\lambda_5\}$ is an interior point. By Proposition \ref{isogformula1}, there exists $c>0$ such that
\[cw_{245}=\dfrac{|\lambda_5-\lambda_4|^2}{r_2^{(245)}}\lambda_2+\dfrac{|\lambda_2-\lambda_5|^2}{r_4^{(245)}}\lambda_4+\dfrac{|\lambda_4-\lambda_2|^2}{r_5^{(245)}}\lambda_5.\]
It suffices to verify that $\textup{Re}(cw_{245})\geq 0$. 

Now, there exist $a,b\geq 1$ such that $\lambda_4=(1-a)\lambda_2+a\lambda_3$ and $\lambda_5=(1-b)\lambda_2+b\lambda_1$. Since $\lambda_1,\lambda_2,\lambda_3$ do not lie on the same line, we can equate the convex weights in the following expression:
\begin{equation}\label{weight1}
\begin{array}{rcl}
\mu_1&=&r_2^{(245)}\lambda_2+r_4^{(245)}\lambda_4+r_5^{(245)}\lambda_5\\
&=&(br_5^{(245)})\lambda_1+[r_2^{(245)}+(1-a)r_4^{(245)}+(1-b)r_5^{(245)}]\lambda_2+(ar_4^{(245)})\lambda_3\\
&=&r_1^{(123)}\lambda_1+r_2^{(123)}\lambda_2+r_3^{(123)}\lambda_3.\end{array}
\end{equation}
By \eqref{weight1} and direct computations, we obtain  \[\begin{array}{rcl}\textup{Re}(cw_{245})&=&a^2b\dfrac{|\lambda_3-\lambda_2|^2}{r_5^{(245)}}\textup{Re}(\lambda_1)+d\textup{Re}(\lambda_2)+ab^2\dfrac{|\lambda_2-\lambda_1|^2}{r_4^{(245)}}\textup{Re}(\lambda_3)\\
&=&a^2b^2\dfrac{|\lambda_3-\lambda_2|^2}{r_1^{(123)}}\textup{Re}(\lambda_1)+d\textup{Re}(\lambda_2)+a^2b^2\dfrac{|\lambda_2-\lambda_1|^2}{r_3^{(123)}}\textup{Re}(\lambda_3)
\end{array}\]
where \[d=\dfrac{|\lambda_5-\lambda_4|^2}{r_2^{(245)}}-\dfrac{b^2(a-1)|\lambda_2-\lambda_1|^2}{r_4^{(245)}}-\dfrac{a^2(b-1)|\lambda_3-\lambda_2|^2}{r_5^{(245)}}.\] Using the identity
\[\lambda_5-\lambda_4=ab(\lambda_1-\lambda_3)+b(a-1)(\lambda_2-\lambda_1)+a(b-1)(\lambda_3-\lambda_2),\] 
 Lemma \ref{keylemma2}, assumption (A2), and \eqref{weight1}, we have that
\[d\textup{Re}(\lambda_2)\geq a^2b^2\dfrac{|\lambda_3-\lambda_1|^2}{r_2^{(123)}}\textup{Re}(\lambda_2)\]
Thus,
\[\textup{Re}(cw_{245})\geq a^2b^2c'\textup{Re}(w_{123})=0\]
for some $c'>0$ due to Proposition \ref{isogformula1} and assumption (A4).
\end{proof}

\begin{lemma}\label{dec28l2}
Let $A\in\C^{5\times 5}$ be normal satisfying (A1)-(A2) and $\mu_1\in W(A)$ satisfying (A3)-(A4). Assume that $\lambda_4,\lambda_5\in[\lambda_3,\lambda_1]$and $\lambda_2,\lambda_4,\lambda_5$ do not lie on the same line. If $\mu_1\in \textup{conv}\{\lambda_2,\lambda_4,\lambda_5\}$, then $\textup{conv}\{w_{245},\lambda_4,\lambda_5\}\subseteq \textup{conv}[\{w_{123}\}\cup (\sigma(A)\ominus\{\lambda_2\})]=\textup{conv}\{w_{123},\lambda_1,\lambda_3\}$.
\end{lemma}
\begin{proof}
We only prove that $w_{245}\in \overline{\textup{RHS}[\lambda_3,w_{123}]}$. The proof for the inclusion $w_{245}\in \overline{\textup{LHS}[\lambda_1,w_{123}]}$ is similar.

If $\lambda_4=\lambda_3$, then $w_{245}$ is on the line obtained by rotating $[\lambda_2,\lambda_3]$ clockwise about $\lambda_3$ at an angle of $\angle(\lambda_5,\lambda_4,\lambda_2)-\angle(\mu_1,\lambda_4,\lambda_2)=\angle(\lambda_1,\lambda_3,\lambda_2)-\angle(\mu_1,\lambda_3,\lambda_2)$. The assertion follows. Assume $\lambda_4\neq \lambda_3$. If $\mu_1\in (\lambda_2,\lambda_4)\cup(\lambda_2,\lambda_5)$, then $w_{245}=\lambda_5$ or $w_{245}=\lambda_4$, and hence the assertion holds in either case. Assume $\mu_1\in \textup{conv}\{\lambda_2,\lambda_4,\lambda_5\}$ is an interior point. By Proposition \ref{isogformula1}, there exists $c>0$ such that
\[cw_{245}=\dfrac{|\lambda_5-\lambda_4|^2}{r_2^{(245)}}\lambda_2+\dfrac{|\lambda_2-\lambda_5|^2}{r_4^{(245)}}\lambda_4+\dfrac{|\lambda_4-\lambda_2|^2}{r_5^{(245)}}\lambda_5.\]
It suffices to verify that $\textup{Re}(cw_{245})\geq 0$. 

Now, there exist $a,b\in [0,1)$ such that $\lambda_4=(1-a)\lambda_1+a\lambda_3$ and $\lambda_5=(1-b)\lambda_1+b\lambda_3$. Since $\lambda_1,\lambda_2,\lambda_3$ do not lie on the same line, we can equate the convex weights in the following expression:
\begin{equation}\label{weight2}
\begin{array}{rcl}
\mu_1&=&r_2^{(245)}\lambda_2+r_4^{(245)}\lambda_4+r_5^{(245)}\lambda_5\\
&=&[(1-a)r_4^{(245)}+(1-b)r_5^{(245)}]\lambda_1+r_2^{(245)}\lambda_2+[ar_4^{(245)}+br_5^{(245)}]\lambda_3\\
&=&r_1^{(123)}\lambda_1+r_2^{(123)}\lambda_2+r_3^{(123)}\lambda_3.\end{array}
\end{equation}
By \eqref{weight2} and direct computations, we obtain \[\textup{Re}(cw_{245})=d_1\textup{Re}(\lambda_1)+(a-b)^2\dfrac{|\lambda_1-\lambda_3|^2}{r_2^{(123)}}\textup{Re}(\lambda_2)+d_3\textup{Re}(\lambda_3)\]
where 
\[d_1=\dfrac{|\lambda_2-(1-b)\lambda_1-b\lambda_3|^2(1-a)}{r_4^{(245)}}+\dfrac{|(1-a)\lambda_1+a\lambda_3-\lambda_2|^2(1-b)}{r_5^{(245)}}\]
and
\[d_3=\dfrac{|\lambda_2-(1-b)\lambda_1-b\lambda_3|^2a}{r_4^{(245)}}+\dfrac{|(1-a)\lambda_1+a\lambda_3-\lambda_2|^2b}{r_5^{(245)}}.\]
Using the identities
\[\lambda_2-(1-b)\lambda_1-b\lambda_3=\lambda_2-\lambda_3+(1-b)(\lambda_3-\lambda_1)=\lambda_2-\lambda_1+b(\lambda_1-\lambda_3)\] and
\[(1-a)\lambda_1+a\lambda_3-\lambda_2=(1-a)(\lambda_1-\lambda_3)+\lambda_3-\lambda_2=a(\lambda_3-\lambda_1)+\lambda_1-\lambda_2,\] 
 Lemma \ref{keylemma2}, assumption (A2), and \eqref{weight2}, we have that
\[d_1\textup{Re}(\lambda_1)\geq (a-b)^2\dfrac{|\lambda_3-\lambda_2|^2}{r_1^{(123)}}\textup{Re}(\lambda_1)\]
and 
\[d_3\textup{Re}(\lambda_3)\geq (a-b)^2\dfrac{|\lambda_2-\lambda_1|^2}{r_3^{(123)}}\textup{Re}(\lambda_3).\]
Thus,
\[\textup{Re}(cw_{245})\geq (a-b)^2c'\textup{Re}(w_{123})=0\]
for some $c'>0$ due to Proposition \ref{isogformula1} and assumption (A4).
\end{proof}

\begin{proposition}\label{dec28l3}
Let $A\in\C^{n\times n}$ be normal satisfying (A1)-(A2) and $\mu_1\in W(A)$ satisfying (A3)-(A4). If $\lambda_2,\lambda_b,\lambda_c$ do not lie on the same line such that $3\leq b<c\leq n$ and $\mu_1\in \textup{conv}\{\lambda_2,\lambda_b,\lambda_c\}$, then $w_{2bc}\in \textup{conv}[\{w_{123}\}\cup (\sigma(A)\ominus\{\lambda_2\})]$.
\end{proposition}
\begin{proof}
If $\mu_1\in (\lambda_2,\lambda_b)\cup(\lambda_2,\lambda_c)$, then $w_{2bc}=\lambda_c$ or $w_{2bc}=\lambda_b$, and hence the assertion holds in either case. Assume $\mu_1\in \textup{conv}\{\lambda_2,\lambda_b,\lambda_c\}$ is an interior point. There exist $\lambda_4'$ and $\lambda_5'$ along the line through $[\lambda_2,\lambda_3]$ and $[\lambda_2,\lambda_1]$ respectively such that $\lambda_b,\lambda_c\in [\lambda_4',\lambda_5']$. By Lemma \ref{dec28l2}, 
\[w_{2bc}\in \textup{conv}\{w_{2bc},\lambda_b,\lambda_c\}\subseteq\textup{conv}\{w_{245}',\lambda_4',\lambda_5'\}.\] Observe that $w_{245}'\in\overline{\textup{RHS}[\lambda_3,w_{123}]}$ due to Lemma \ref{dec28l1} while $\lambda_4',\lambda_5'\in \overline{\textup{RHS}[\lambda_3,w_{123}]}$ due to assumption (A2). The assertion follows.
\end{proof}

\begin{lemma}\label{mainlemma}
Let $A\in \C^{4\times 4}$ be normal with distinct eigenvalues $\lambda_1,\ldots,\lambda_4$ not lying on the same line and arranged in a counterclockwise orientation with respect to $\textup{trace}(A)/4$ such that no eigenvalue is in the interior of $W(A)$. Let $\mu_1\notin \partial W(A)$, $\mu_1\notin\Lambda_2(A)$, and $\mu_1\in \textup{conv}\{\lambda_1,\lambda_2,\lambda_3\}\cap\textup{conv}\{\lambda_1,\lambda_2,\lambda_4\}$. Let $w$ be the intersection of the lines through $[\lambda_1,w_{123}]$ and $[\lambda_2,w_{124}]$. Then
\[{\cal B}_A(\mu_1)\subseteq \textup{conv}\{w,w_{123},w_{124},\lambda_3,\lambda_4\},\]
and in particular, $\lambda_3, \lambda_4\in\overline{\textup{RHS}[w_{123},w_{124}]}$. Moreover, if either $\lambda_3\in (\lambda_2,\lambda_4)$ or $\lambda_4\in (\lambda_3,\lambda_1)$, then Theorem \ref{main} holds.
\end{lemma}
\begin{proof}
Let $\mu_2\in {\cal B}_A(\mu_1)$ and $\Lambda_2(A)=\{x\}$, where $x$ is the intersection of $[\lambda_3,\lambda_1]$ and $[\lambda_2,\lambda_4]$. The proof of the assertions $\mu_2\in \overline{\textup{LHS}[\lambda_1,w_{123}]}$ and $\mu_2\in \overline{\textup{RHS}[\lambda_3,w_{123}]}$ when $\mu_1\in\textup{conv}\{\lambda_1,\lambda_2,x\}$ is an interior point is the same as when $\mu_1\in (\lambda_2,x)$ since $\mu_1$ is still an interior point of $\textup{conv}\{\lambda_1,\lambda_2,\lambda_3\}$. 
Similarly, the proof of the assertions $\mu_2\in\overline{\textup{RHS}[\lambda_2,w_{124}]}$ and $\mu_2\in \overline{\textup{LHS}[\lambda_4,w_{124}]}$ when $\mu_1\in \textup{conv}\{\lambda_1,\lambda_2,x\}$ is an interior point is the same as when $\mu_1\in (\lambda_1,x)$ since $\mu_1$ is an interior point of $\textup{conv}\{\lambda_1,\lambda_2,\lambda_4\}$. Hence, we assume $\mu_1\in\textup{conv}\{\lambda_1,\lambda_2,x\}$ is an interior point.

Since $\mu_2\in W(A)$, it is clear that $\mu_2\in \overline{\textup{LHS}[\lambda_3,\lambda_4]}$. We only prove that $\mu_2\in \overline{\textup{LHS}[\lambda_1,w_{123}]}$, that is,
\begin{equation}\label{lhs3}
\textup{Im}\left(\dfrac{\mu_2-\lambda_1}{w_{123}-\lambda_1}\right)\geq 0.
\end{equation}

\begin{center}
\begin{tikzpicture}[thick]
\coordinate[label=above:$\lambda_1$] (a) at (0.2,4);
\coordinate[label=left:$\lambda_2$] (b) at (0,0);
\coordinate[label=right:$\lambda_3$] (c) at (4,0);
\coordinate[label=right:$\lambda_4$] (d) at (4.4,3.8);
\coordinate[label=below:$\mu_1$] (z) at (0.6,0.64);
\coordinate[label=left:$w_{123}$] (w) at (1.8390,1.5590);
\coordinate[label=below:$w_{124}$] (w2) at (0.6119,3.9096);

\coordinate (zonl1l2) at (0.037296,0.74592);
\coordinate (zonl2l3) at (0.67619,0);
\coordinate (zonl3l1) at (1.986755,2.1192);
\coordinate (wonl1l2) at (0.1392617,2.785233);
\coordinate (wonl2l3) at (2.8857868,0);
\coordinate (w2onl1l2) at (0.196081,3.92162);
\coordinate (w2onl4l1) at (0.622896211,3.9798583);

\draw (a)--(b)--(c)--(d)--cycle;

\fill[purple] (a)  circle[radius=1.5pt];
\fill[purple] (b)  circle[radius=1.5pt];
\fill[purple] (c)  circle[radius=1.5pt];
\fill[purple] (d)  circle[radius=1.5pt];
\fill[green] (z)  circle[radius=1.5pt];
\fill[red] (w)  circle[radius=1.5pt];
\fill[red] (w2)  circle[radius=1.5pt];


\draw[->,orange] (c)--(wonl1l2);
\draw[fill=gray,opacity=0.2,text opacity=1] (c)--(wonl1l2)--(a)--(d)--cycle;
\draw[->,orange] (a)--(wonl2l3);
\draw[fill=gray,opacity=0.2,text opacity=1] (a)--(wonl2l3)--(c)--(d)--cycle;

\draw[->,orange] (b)--(w2onl4l1);
\draw[fill=gray,opacity=0.2,text opacity=1] (b)--(w2onl4l1)--(d)--(c)--cycle;

\draw[->,orange] (d)--(w2onl1l2);
\draw[fill=gray,opacity=0.2,text opacity=1] (d)--(w2onl1l2)--(b)--(c)--cycle;

\end{tikzpicture}
\end{center}
Analogous arguments show that $\mu_2\in\overline{\textup{RHS}[\lambda_2,w_{124}]}$, $\mu_2\in \overline{\textup{LHS}[\lambda_4,w_{124}]}$, and $\mu_2\in \overline{\textup{RHS}[\lambda_3,w_{123}]}$. By \cite[Theorem 3]{rC13}, there exists $c\geq0$ such that 
\begin{equation}\label{isog1}
c=\dfrac{(\mu_1-\lambda_1)(w_{123}-\lambda_1)}{(\lambda_2-\lambda_1)(\lambda_3-\lambda_1)}.
\end{equation}
Since $\mu_1\in \textup{conv}\{\lambda_1,\lambda_2,\lambda_3\}$ is an interior point, $c>0$.

By setting $z=\lambda_1$ in \eqref{nec2}, there exist some $p_{ij}\in [0,1]$ with $\displaystyle\sum_{i<j}p_{ij}=1$ such that \[(\mu_1-\lambda_1)(\mu_2-\lambda_1)=\displaystyle\sum_{i<j}p_{ij}(\lambda_i-\lambda_1)(\lambda_j-\lambda_1)\]
which can be simplified as
\begin{equation}\label{conv1}
p_{23}(\lambda_2-\lambda_1)(\lambda_3-\lambda_1)+p_{24}(\lambda_2-\lambda_1)(\lambda_4-\lambda_1)+p_{34}(\lambda_3-\lambda_1)(\lambda_4-\lambda_1).\end{equation}
By \eqref{isog1} and \eqref{conv1}, we obtain
\[\begin{array}{rcl}
\dfrac{\mu_2-\lambda_1}{w_{123}-\lambda_1}&=&\dfrac{1}{c}\dfrac{(\mu_1-\lambda_1)(\mu_2-\lambda_1)}{(\lambda_2-\lambda_1)(\lambda_3-\lambda_1)}\\
&=&\dfrac{1}{c}\left(p_{23}+p_{24}\dfrac{\lambda_4-\lambda_1}{\lambda_3-\lambda_1}+p_{34}\dfrac{\lambda_4-\lambda_1}{\lambda_2-\lambda_1}\right).
\end{array}\]
Taking the imaginary part gives us
\[\textup{Im}\left(\dfrac{\mu_2-\lambda_1}{w_{123}-\lambda_1}\right)=\dfrac{p_{24}}{c}\textup{Im}\left(\dfrac{\lambda_4-\lambda_1}{\lambda_3-\lambda_1}\right)+\dfrac{p_{34}}{c}\textup{Im}\left(\dfrac{\lambda_4-\lambda_1}{\lambda_2-\lambda_1}\right)\geq 0\]
due to $c$ being positive and the eigenvalues having a counterclockwise orientation. This proves \eqref{lhs3}.

Since $w_{123},w_{124}\in {\cal B}_A(\mu_1)$, the first part applies to $w_{123}, w_{124}$. Note that $\lambda_3\in \overline{\textup{RHS}[w_{123},w_{124}]}$ if and only if $w_{124}\in \overline{\textup{RHS}[\lambda_3,w_{123}]}$ which is true by the first part. Similarly, $w_{123}\in \overline{\textup{LHS}[\lambda_4,w_{124}]}$ since $\lambda_4\in \overline{\textup{RHS}[w_{123},w_{124}]}$.

Finally, if $\lambda_3\in (\lambda_2,\lambda_4)$ or $\lambda_4\in (\lambda_3,\lambda_1)$, then $w=w_{123}$ or $w=w_{124}$, respectively. Thus, $\textup{conv}[{\cal B}_A(\mu_1)]=\textup{conv}\{w_{123},w_{124},\lambda_3,\lambda_4\}={\cal R}_A(\mu_1)$.
\end{proof}

\begin{corollary}\label{dec29}
Let $A\in\C^{n\times n}$ be normal satisfying (A1)-(A2) and $\mu_1\in W(A)$ satisfying (A3)-(A4). If $\lambda_1,\lambda_2,\lambda_c$ do not lie on the same line such that $3\leq c\leq n$ and $\mu_1\in \textup{conv}\{\lambda_1,\lambda_2,\lambda_c\}$, then $w_{12c}\in \textup{conv}[\{w_{123}\}\cup (\sigma(A)\ominus\{\lambda_2\})]$.
\end{corollary}

\begin{proposition}\label{realpartmup1v2}
Let $A\in\C^{n\times n}$ be normal satisfying (A1)-(A2) and $\mu_1\in W(A)$ satisfying (A3)-(A4). If $t\in {\cal C}_\Lambda(\mu_1)$, then \[{\cal B}_A(\mu_1,t)\subseteq \textup{conv}[\{w_{123}\}\cup(\sigma(A)\ominus\{\lambda_2\})].\]
\end{proposition}
\begin{proof}
We prove this by induction on $n$. Let $t\in {\cal C}_\Lambda(\mu_1)$. Suppose $t$ has zero entries. Let $J:=\{j:t_j=0\}$, $S:=\{\lambda_j: j\in J\}$, and $\Lambda':=\textup{diag}(\lambda_j)_{j\notin J}$. By Proposition \ref{withzero}, 
\[{\cal B}_A(\mu_1,t)=\textup{conv}[{\cal B}_{\Lambda'}(\mu_1,s)\cup W(S)]\]
where $s=[t_j]_{j\notin J}$. By the induction hypothesis, 
\[{\cal B}_{\Lambda'}(\mu_1,s)\subseteq{\cal B}_{\Lambda'}(\mu_1)\subseteq\textup{conv}[\{w_{abc}\}\cup T]\]
where $w_{abc}$ is the isogonal conjugate of $\mu_1$ with respect to $\textup{conv}\{\lambda_a,\lambda_b,\lambda_c\}$ and $T\subseteq \sigma(A)$. Due to assumption (A3), $\lambda_2\notin T$ and $\lambda_2\in \{\lambda_a,\lambda_b,\lambda_c\}$. Hence, \[T\subseteq \textup{conv}[\{w_{123}\}\cup(\sigma(A)\ominus\{\lambda_2\})].\] Moreover, observe that $w_{abc}\in\textup{conv}[\{w_{123}\}\cup(\sigma(A)\ominus\{\lambda_2\})] $ due to Proposition \ref{dec28l3} and Corollary \ref{dec29}. The assertion follows. 

If $t$ has all positive entries, then ${\cal B}_A(\mu_1,t)\subseteq \textup{conv}[\{w_{123}\}\cup(\sigma(A)\ominus\{\lambda_2\})]$ due to Proposition \ref{realpartmup1}.
\end{proof}
\section{$\mu_1\notin \Lambda_2(A)$: ${\cal B}_A(\mu_1)\subseteq \overline{\textup{RHS}[w_{123},w_{12n}]}$}


\begin{proposition}\label{A2}
Let $A\in \C^{n\times n}$ be normal with distinct eigenvalues $\lambda_1,\ldots,\lambda_n$ not lying on the same line and arranged in a counterclockwise orientation with respect to $\textup{trace}(A)/n$ such that no eigenvalue is in the interior of $W(A)$. Suppose $\textup{conv}\{\lambda_1,\lambda_2,\lambda_3,\lambda_n\}$ determines a $4$-gon. If $\mu_1\notin\partial W(A)$, $\mu_1\notin\Lambda_2(A)$, and $\mu_1\in \textup{conv}\{\lambda_1,\lambda_2,\lambda_3\}\cap\textup{conv}\{\lambda_1,\lambda_2,\lambda_n\}$, then $\lambda_1,\lambda_2\in\overline{\textup{LHS}[w_{123},w_{12n}]}$ and $\lambda_j\in \overline{\textup{RHS}[w_{123},w_{12n}]}$ for all $\lambda_j\neq \lambda_1,\lambda_2$.
\end{proposition}
\begin{proof}
Let $x$ be the intersection of $[\lambda_3,\lambda_1]$ and $[\lambda_2,\lambda_n]$. By \cite[Theorem 3]{rC13}, there exist $a,b\geq0$ for which 
\[w_{123}-\lambda_1=a\dfrac{(\lambda_2-\lambda_1)(\lambda_3-\lambda_1)}{\mu_1-\lambda_1}\]
and
\[w_{12n}-\lambda_1=b\dfrac{(\lambda_2-\lambda_1)(\lambda_n-\lambda_1)}{\mu_1-\lambda_1}.\]
If $a=0$, then $w_{123}=\lambda_1$, which implies $\mu_1\in [\lambda_2,\lambda_3]\subseteq \partial W(A)$, a contradiction. Hence, $a>0$. Now, \[\dfrac{w_{12n}-w_{123}}{\lambda_1-w_{123}}=1-\dfrac{b}{a}\cdot \dfrac{\lambda_n-\lambda_1}{\lambda_3-\lambda_1}.\]
Taking the imaginary part gives us
\[\textup{Im}\left(\dfrac{w_{12n}-w_{123}}{\lambda_1-w_{123}}\right)=-\dfrac{b}{a}\textup{Im}\left(\dfrac{\lambda_n-\lambda_1}{\lambda_3-\lambda_1}\right)\leq 0\]
since $a,b\geq0$ and $\textup{conv}\{\lambda_1,\lambda_2,\lambda_3,\lambda_n\}$ has a counterclockwise orientation. This proves that $\lambda_1\in\overline{\textup{LHS}[w_{123},w_{12n}]}$. An analogous computation reveals that $\lambda_2\in\overline{\textup{LHS}[w_{123},w_{12n}]}$.

Finally, by applying Lemma \ref{mainlemma} to $\Lambda':=\textup{diag}(\lambda_1,\lambda_2,\lambda_3,\lambda_n)$, we have that $\lambda_3,\lambda_n\in \overline{\textup{RHS}[w_{123},w_{12n}]}$. For all $\lambda_j\neq\lambda_1,\lambda_2$, it follows that $\lambda_j\in \overline{\textup{RHS}[w_{123},w_{12n}]}$ since the eigenvalues determine a counterclockwise orientation and $\lambda_1,\lambda_2\in\overline{\textup{LHS}[w_{123},w_{12n}]}$. \end{proof}

Proposition \ref{A2} guarantees that after doing a rotation or translation argument, we may assume the following:
\begin{enumerate}
\item[(B1)] Given $n\geq 4$, let $A\in\C^{n\times n}$ be normal with distinct eigenvalues $\lambda_1,\ldots,\lambda_n$ not lying on the same line and arranged in a counterclockwise orientation with respect to $\textup{trace}(A)/n$ such that no eigenvalue is in the interior of $W(A)$.
\item[(B2)] $\textup{conv}\{\lambda_1,\lambda_2,\lambda_3,\lambda_n\}$ determines a $4$-gon. Moreover, $\textup{Re}(\lambda_1), \textup{Re}(\lambda_2)\leq0$ and $\textup{Re}(\lambda_j)\geq 0$ for all $\lambda_j\neq \lambda_2,\lambda_3$. At least one eigenvalue in $\{\lambda_1,\lambda_2\}$ has strictly negative real part.
\item[(B3)] $\mu_1\notin \partial W(A)$, $\mu_1\notin\Lambda_2(A)$, and $\mu_1\in \textup{conv}\{\lambda_1,\lambda_2,\lambda_3\}\cap \textup{conv}\{\lambda_1,\lambda_2,\lambda_n\}$.
\item[(B4)] $\textup{Re}(w_{123})=\textup{Re}(w_{12n})=0$ where $w_{12j} $ is the isogonal conjugate of $\mu_1$ with respect to $\textup{conv}\{\lambda_1,\lambda_2,\lambda_j\}$ for $j=3,n$.
\end{enumerate}

\begin{lemma}\label{coll}
Let $A\in \C^{5\times 5}$ be normal satisfying (B1)-(B2) and $\mu_1\in W(A)$ satisfying (B3)-(B4). Assume $\lambda_4\in (\lambda_3,\lambda_5)$. If $\mu_1\in\textup{conv}\{\lambda_1,\lambda_2,\lambda_4\}$, then $w_{124}\in\overline{\textup{RHS}[w_{123},w_{125}]}$.
\end{lemma}
\begin{proof} Let $x$ be the intersection of $[\lambda_3,\lambda_1]$ and $[\lambda_2,\lambda_5]$. If $\mu_1\in(\lambda_1,x)\cup(\lambda_2,x)$, then the assertion follows from Lemma \ref{mainlemma}. Assume $\mu_1\in\textup{conv}\{\lambda_1,\lambda_2,x\}$ is an interior point. By Proposition \ref{isogformula1}, $w_{124}$ satisfies
\[ c w_{124}=\dfrac{|\lambda_4-\lambda_2|^2}{r_1^{(124)}}\lambda_1+\dfrac{|\lambda_1-\lambda_4|^2}{r_2^{(124)}}\lambda_2+\dfrac{|\lambda_2-\lambda_1|^2}{r_4^{(124)}}\lambda_4\]
for some $c>0$. It suffices to show that $\textup{Re}(c w_{124})\geq 0$. 

Let $\lambda_4=(1-a)\lambda_3+a\lambda_5$ for some $a\in (0,1)$. Then
\[\begin{array}{rcl}
\mu_1&=&r_1^{(124)}\lambda_1+r_2^{(124)}\lambda_2+r_4^{(124)}\lambda_4\\
&=&r_1^{(124)}\lambda_1+r_2^{(124)}\lambda_2+[(1-a)r_4^{(124)}]\lambda_3+(ar_4^{(124)})\lambda_5.\\
\end{array}\]
By Proposition \ref{givenzprop2}, there exists $k\in (0,1)$ such that
\begin{equation}\label{alpha} [r_1^{(124)}\ r_2^{(124)}\ (1-a)r_4^{(124)}\ ar_4^{(124)}]=(1-k)r^{(123)}+kr^{(125)}\end{equation}
where we take $r^{(125)}=r_1^{(125)}e_1+r_2^{(125)}e_2+r_5^{(125)}e_4$.
By \eqref{alpha} and direct computations, we obtain
\[\textup{Re}(c w_{124})=d_1\textup{Re}(\lambda_1)+d_2\textup{Re}(\lambda_2)+\dfrac{(1-a)^2|\lambda_2-\lambda_1|^2}{(1-k)r_3^{(123)}}\textup{Re}(\lambda_3)+\dfrac{a^2|\lambda_2-\lambda_1|^2}{kr_5^{(125)}}\textup{Re}(\lambda_5)\]
where 
\[d_1=\dfrac{|(1-a)(\lambda_3-\lambda_2)+a(\lambda_5-\lambda_2)|^2}{(1-k)r_1^{(123)}+kr_1^{(125)}}\]
and
\[d_2=\dfrac{|(1-a)(\lambda_1-\lambda_3)+a(\lambda_1-\lambda_5)|^2}{(1-k)r_2^{(123)}+kr_2^{(125)}}.\]
By Lemma \ref{keylemma2} and assumption (B2),
\begin{equation}\label{coeff1}
d_1\textup{Re}(\lambda_1)\geq
\left[\dfrac{(1-a)^2|\lambda_3-\lambda_2|^2}{(1-k)r_1^{(123)}}+\dfrac{a^2|\lambda_5-\lambda_2|^2}{kr_1^{(125)}}\right]\textup{Re}(\lambda_1)
\end{equation}
and
\begin{equation}\label{coeff2}d_2\textup{Re}(\lambda_2)\geq\left[ \dfrac{(1-a)^2|\lambda_1-\lambda_3|^2}{(1-k)r_2^{(123)}}+\dfrac{a^2|\lambda_1-\lambda_5|^2}{kr_2^{(125)}}\right]\textup{Re}(\lambda_2).\end{equation}

Since $\mu_1$ is an interior point of $\textup{conv}\{\lambda_1,\lambda_2,x\}$, Proposition \ref{isogformula1} implies 
\[c_3w_{123}=\dfrac{|\lambda_3-\lambda_2|^2}{r_1^{(123)}}\lambda_1+\dfrac{|\lambda_1-\lambda_3|^2}{r_2^{(123)}}\lambda_2+\dfrac{|\lambda_2-\lambda_1|^2}{r_3^{(123)}}\lambda_3\]
and 
\[c_5w_{125}=\dfrac{|\lambda_5-\lambda_2|^2}{r_1^{(125)}}\lambda_1+\dfrac{|\lambda_1-\lambda_5|^2}{r_2^{(125)}}\lambda_2+\dfrac{|\lambda_2-\lambda_1|^2}{r_5^{(125)}}\lambda_5\]
for some $c_3, c_5>0$.

Statements \eqref{coeff1}-\eqref{coeff2} and (B4) imply
\[\begin{array}{rcl}
\vspace{0.2cm}
\textup{Re}(c w_{124})&\geq &\left[\dfrac{(1-a)^2|\lambda_3-\lambda_2|^2}{(1-k)r_1^{(123)}}+\dfrac{a^2|\lambda_5-\lambda_2|^2}{kr_1^{(125)}}\right]\textup{Re}(\lambda_1)+\\
\vspace{0.2cm}
& & \left[\dfrac{(1-a)^2|\lambda_1-\lambda_3|^2}{(1-k)r_2^{(123)}}+\dfrac{a^2|\lambda_1-\lambda_5|^2}{kr_2^{(125)}}\right]\textup{Re}(\lambda_2)+\\
\vspace{0.2cm}
& &\dfrac{(1-a)^2|\lambda_2-\lambda_1|^2}{(1-k)r_3^{(123)}}\textup{Re}(\lambda_3)+\dfrac{a^2|\lambda_2-\lambda_1|^2}{kr_5^{(125)}}\textup{Re}(\lambda_5)\\
\vspace{0.2cm}
&=&\dfrac{(1-a)^2c_3}{1-k}\textup{Re}(w_{123})+\dfrac{a^2c_5}{k}\textup{Re}(w_{125})\\
&=&0.
\end{array}\]
\end{proof}

\begin{lemma}\label{anotherlemma}
Let $A\in \C^{5\times 5}$ be normal satisfying (B1)-(B2) and $\mu_1\in W(A)$ satisfying (B3)-(B4). Assume $\lambda_4$ is along the line through $[\lambda_2,\lambda_3]$. If $\mu_1\in\textup{conv}\{\lambda_1,\lambda_2,\lambda_4\}$, then $w_{124}\in\textup{conv}\{w_{123},w_{125},\lambda_3,\lambda_5\}$.
\end{lemma}
\begin{proof}
By Corollary \ref{dec29}, it suffices to prove that $w_{124}\in \overline{\textup{RHS}[w_{123},w_{125}]}$. By assumption (B3), $\mu_1\notin (\lambda_2,\lambda_4)\cup(\lambda_1,\lambda_4)$, and thus $\mu_1\in \textup{conv}\{\lambda_1,\lambda_2,\lambda_4\}$ is an interior point. By Proposition \ref{isogformula1}, there exists $c>0$ such that
\[cw_{124}=\dfrac{|\lambda_4-\lambda_2|^2}{r_1^{(124)}}\lambda_1+\dfrac{|\lambda_1-\lambda_4|^2}{r_2^{(124)}}\lambda_2+\dfrac{|\lambda_2-\lambda_1|^2}{r_4^{(124)}}\lambda_4.\] It suffices to verify that $\textup{Re}(cw_{124})\geq 0$.

Write $\lambda_4=(1-a)\lambda_2+a\lambda_3$, for some $a>1$. Since $\lambda_1,\lambda_2,\lambda_3$ do not lie on the same line, we can equate the convex weights in the following expression:
\begin{equation}\label{mu1again}
\begin{array}{rcl}
\mu_1&=&r_1^{(124)}\lambda_1+r_2^{(124)}\lambda_2+r_4^{(124)}\lambda_4\\
&=&r_1^{(124)}\lambda_1+[r_2^{(124)}+(1-a)r_4^{(124)}]\lambda_2+(ar_4^{(124)})\lambda_3\\
&=&r_1^{(123)}\lambda_1+r_2^{(123)}\lambda_2+r_3^{(123)}\lambda_3.
\end{array}
\end{equation}
By \eqref{mu1again} and direct computations, we obtain
\[\textup{Re}(cw_{124})=a^2\dfrac{|\lambda_3-\lambda_2|^2}{r_1^{(123)}}\textup{Re}(\lambda_1)+d\textup{Re}(\lambda_2)+a^2\dfrac{|\lambda_2-\lambda_1|^2}{r_2^{(123)}}\]
where 
\[d=\dfrac{|(1-a)(\lambda_1-\lambda_2)+a(\lambda_1-\lambda_3)|^2}{r_2^{(123)}+(a-1)r_4^{(124)}}-\dfrac{(a-1)|\lambda_2-\lambda_1|^2}{r_4^{(124)}}\]
due to \eqref{mu1again}. Lemma \ref{keylemma2} and assumption (B2) guarantee that
\[d\textup{Re}(\lambda_2)\geq a^2\dfrac{|\lambda_1-\lambda_3|^2}{r_2^{(123)}}\textup{Re}(\lambda_2).\]
Thus
\[\textup{Re}(cw_{124})\geq a^2c'\textup{Re}(w_{123})\geq 0\]
for some $c'>0$ due to Proposition \ref{isogformula1} and assumption (B4).
\end{proof}

\begin{proposition}\label{isogrhs}
Let $A\in \C^{n\times n}$ be normal satisfying (B1)-(B2) and $\mu_1\in W(A)$ satisfying (B3)-(B4). If $\mu_1\in \textup{conv}\{\lambda_1,\lambda_2,\lambda_c\}$ such that $3\leq c\leq n$, then $w_{12c}\in \textup{conv}\{w_{123},w_{12n},\lambda_3,\ldots,\lambda_n\}$.
\end{proposition}
\begin{proof} Let $x$ be the intersection of $[\lambda_3,\lambda_1]$ and $[\lambda_2,\lambda_n]$. If $\mu_1\in (\lambda_1,x)\cup(\lambda_2,x)$, then the assertion holds due to Lemma \ref{mainlemma}. Assume $\mu_1\in \textup{conv}\{\lambda_1,\lambda_2,x\}$ is an interior point. 

By Corollary \ref{dec29}, it suffices to show that $w_{12c}\in \overline{\textup{RHS}[w_{123},w_{12n}]}$. There exist $\lambda_3'$ and $\lambda_n'$ along the line through $[\lambda_2,\lambda_3]$ and $[\lambda_1,\lambda_n]$ respectively such that $\lambda_c\in [\lambda_3',\lambda_n']$. Consider $\Lambda':=\textup{diag}(\lambda_1,\lambda_2,\lambda_3',\lambda_c,\lambda_n')$. 
By Corollary \ref{dec29} and Lemma \ref{coll}, $w_{12c}\in \textup{conv}\{w_{123}',w_{12n}',\lambda_3',\lambda_n'\}$ where $w_{12j}'$ is the isogonal conjugate of $\mu_1$ with respect to $\textup{conv}\{\lambda_1,\lambda_2,\lambda_j'\}$. By assumption (A2), $\lambda_3',\lambda_n'\in\overline{\textup{RHS}[w_{123},w_{12n}]}$. Note that $w_{123}',w_{12n}'\in\overline{\textup{RHS}[w_{123},w_{12n}]}$ due to Lemma \ref{anotherlemma}. It follows that
$w_{12c} 
\in\overline{\textup{RHS}[w_{123},w_{12n}]}$.\end{proof}

\begin{proposition}\label{realpartmup2}
Let $A\in \C^{n\times n}$ be normal satisfying (B1)-(B2) and $\mu_1\in W(A)$ satisfying (B3)-(B4). If $t\in{\cal C}_\Lambda(\mu_1)$ has all positive entries, then \[{\cal B}_A(\mu_1,t)\subseteq\textup{conv}\{w_{123},w_{12n},\lambda_3,\ldots,\lambda_n\}.\]
\end{proposition}
\begin{proof}
By Proposition \ref{realpartmup1}, it suffices to prove that ${\cal B}_A(\mu_1,t)\subseteq \overline{\textup{RHS}[w_{123},w_{12n}]}$. 

Elements of ${\cal B}_A(\mu_1,t)$ are of the form $\frac{v^*\Lambda v}{v^*v}$ for all nonzero $v\in \{\sqrt{t},\Lambda\sqrt{t}\}^\perp$. By (B1)-(B4), it suffices to prove $\textup{Re}(v^*\Lambda v)\geq 0$.  A basis for $\{\sqrt{t},\Lambda\sqrt{t}\}^\perp$ is given by
\[f_j=\dfrac{\overline{\lambda_j-\lambda_2}}{\sqrt{t_1}}e_1+\dfrac{\overline{\lambda_1-\lambda_j}}{\sqrt{t_2}}e_2+\dfrac{\overline{\lambda_2-\lambda_1}}{\sqrt{t_j}}e_j\]
for $j=3,\ldots,n$ where $e_k$ is the $k^{th}$ standard basis vector in $\C^n$. The vector $t$ can be written as a convex combination $t=\displaystyle\sum_{j=3}^n p_jr^{(12j)}$
where $p_j>0$ and $\displaystyle\sum_{j=3}^n p_j=1$. 

Observe that $\textup{Re}(v^*\Lambda v)\geq 0$ for all $v\in \{\sqrt{t},\Lambda\sqrt{t}\}^\perp$ if and only if $Z(t)\geq 0$, as defined in \eqref{psd}. If $x=[x_3\ \ldots\ x_n]^T\in\C^{n-2}$, then
\[x^*Z(t) x=\dfrac{\left|\displaystyle\sum_{j=3}^n\overline{x_j}(\lambda_j-\lambda_2)\right|^2}{t_1}\textup{Re}(\lambda_1)+\dfrac{\left|\displaystyle\sum_{j=3}^n\overline{x_j}(\lambda_1-\lambda_j)\right|^2}{t_2}\textup{Re}(\lambda_2)+\displaystyle\sum_{j=3}^n\dfrac{|\overline{x_j}(\lambda_2-\lambda_1)|^2}{t_j}\textup{Re}(\lambda_j).\]
For some $c_j>0$, Proposition \ref{isogformula1} guarantees that the isogonal conjugate of $\mu_1$ with respect to $\textup{conv}\{\lambda_1,\lambda_2,\lambda_j\}$ satisfies
\[c_jw_{12j}=\frac{|\lambda_j-\lambda_2|^2}{r_1^{(12j)}}\lambda_1+\frac{|\lambda_1-\lambda_j|^2}{r_2^{(12j)}}\lambda_2+\frac{|\lambda_2-\lambda_1|^2}{r_j^{(12j)}}\lambda_j\]
for all $j=3,\ldots,n$. Note that
\[\begin{array}{rcl}
\dfrac{|\overline{x_j}(\lambda_2-\lambda_1)|^2}{t_j}\lambda_j&=&\dfrac{|\overline{x_j}|^2}{p_j}\cdot \dfrac{|\lambda_2-\lambda_1|^2}{r_j^{(12j)}}\lambda_j\\
&=&\dfrac{|\overline{x_j}|^2c_j}{p_j}w_{12j}- \dfrac{|\overline{x_j}(\lambda_j-\lambda_2)|^2}{p_jr_1^{(12j)}}\lambda_1-\dfrac{|\overline{x_j}(\lambda_1-\lambda_j)|^2}{p_jr_2^{(12j)}}\lambda_2,
\end{array}
\]
for all $j=3,\ldots,n$, and hence, $\displaystyle\sum_{j=3}^n\dfrac{|\overline{x_j}(\lambda_2-\lambda_1)|^2}{t_j}\lambda_j$ is equal to
\[\displaystyle\sum_{j=3}^n\dfrac{|\overline{x_j}|^2c_j}{p_j}w_{12j}-\displaystyle\sum_{j=3}^n 
\dfrac{|\overline{x_j}(\lambda_j-\lambda_2)|^2}{p_jr_1^{(12j)}}\lambda_1-\displaystyle\sum_{j=3}^n\dfrac{|\overline{x_j}(\lambda_1-\lambda_j)|^2}{p_jr_2^{(12j)}}\lambda_2.\]
Thus, 
\[x^*Z(t)x=d_1\textup{Re}(\lambda_1)+d_2\textup{Re}(\lambda_2)+\displaystyle\sum_{j=3}^n\dfrac{|\overline{x_j}|^2c_j}{p_j}\textup{Re}(w_{12j})\]
where 
\[d_1=\dfrac{\left|\displaystyle\sum_{j=3}^n\overline{x_j}(\lambda_j-\lambda_2)\right|^2}{t_1}-\displaystyle\sum_{j=3}^n 
\dfrac{|\overline{x_j}(\lambda_j-\lambda_2)|^2}{p_jr_1^{(12j)}}\]
and
\[d_2=\dfrac{\left|\displaystyle\sum_{j=3}^n\overline{x_j}(\lambda_1-\lambda_j)\right|^2}{t_2}-\displaystyle\sum_{j=3}^n\dfrac{|\overline{x_j}(\lambda_1-\lambda_j)|^2}{p_jr_2^{(12j)}}.\] Note that $t_k=\displaystyle\sum_{j=3}^np_jr_k^{(12j)}$ for $k=1,2$, and so Lemma \ref{keylemma2} guarantees that \[d_1,d_2\leq 0.\] Hence,
\[d_1\textup{Re}(\lambda_1)+d_2\textup{Re}(\lambda_2) \geq 0\]
due to assumption (B2). It follows that
\[x^*Z(t)x=d_1\textup{Re}(\lambda_1)+d_2\textup{Re}(\lambda_2)+\displaystyle\sum_{j=3}^n\dfrac{|\overline{x_j}|^2c_j}{p_j}\textup{Re}(w_{12j})\geq 0\]
since $\textup{Re}(w_{12j})\geq 0$ for all $j=3,\ldots,n$ due to Proposition \ref{isogrhs}.
\end{proof}

\begin{proposition}\label{realpartmup2v2}
Let $A\in\C^{n\times n}$ be normal satisfying (B1)-(B2) and $\mu_1\in W(A)$ satisfying (B3)-(B4). If $t\in {\cal C}_\Lambda(\mu_1)$, then 
\[{\cal B}_A(\mu_1,t)\subseteq \textup{conv}\{w_{123},w_{12n},\lambda_3,\ldots,\lambda_n\}.\]
\end{proposition}
\begin{proof}
We prove this by induction on $n$. Let $t\in {\cal C}_\Lambda(\mu_1)$. Suppose $t$ has zero entries. Let $J:=\{j:t_j=0\}$, $S:=\{\lambda_j: j\in J\}$, and $\Lambda':=\textup{diag}(\lambda_j)_{j\notin J}$. By Proposition \ref{withzero}, 
\[{\cal B}_A(\mu_1,t)=\textup{conv}[{\cal B}_{\Lambda'}(\mu_1,s)\cup W(S)]\]
where $s=[t_j]_{j\notin J}$. By the induction hypothesis, 
\[{\cal B}_{\Lambda'}(\mu_1,s)\subseteq{\cal B}_{\Lambda'}(\mu_1)\subseteq\textup{conv}[\{w_{abc},w_{def}\}\cup T]\]
where $w_{abc},w_{def}$ are the isogonal conjugates of $\mu_1$ with respect to $\textup{conv}\{\lambda_a,\lambda_b,\lambda_c\}$ and $\textup{conv}\{\lambda_d,\lambda_e,\lambda_f\}$ and $T\subseteq \sigma(A)$. Due to assumption (B3), $\lambda_1,\lambda_2\notin T$, $\lambda_1,\lambda_2\in \{\lambda_a,\lambda_b,\lambda_c\}$, and $\lambda_1,\lambda_2\in \{\lambda_d,\lambda_e,\lambda_f\}$. Hence, \[T\subseteq \textup{conv}\{w_{123},w_{12n},\lambda_3,\ldots,\lambda_n\}.\] Moreover, observe that $w_{abc}=w_{12c},w_{def}=w_{12f}\in\textup{conv}\{w_{123},w_{12n},\lambda_3,\ldots,\lambda_n\} $ due to Proposition \ref{isogrhs}. The assertion follows. 

If $t$ has all positive entries, then ${\cal B}_A(\mu_1,t)\subseteq \textup{conv}\{w_{123},w_{12n},\lambda_3,\ldots,\lambda_n\} $ due to Proposition \ref{realpartmup2}.
\end{proof}
\section{Proof of the main result}

\begin{proof}[Proof of Theorem \ref{main}]
As discussed in Section 3, we consider the following cases:

\noindent\textbf{Case 1}: $\mu_1\in \Lambda_2(A)$.\\

\noindent\textbf{Case 2}: $\mu_1\in \partial W(A)\cap [W(A)\setminus\Lambda_2(A)]$.\\

\noindent\textbf{Case 3}: $\mu_1\in W(A)\setminus [ \Lambda_2(A)\cup \partial W(A)]$ and there exists unique $\lambda_a\in \sigma(A)\setminus \Lambda_2(A)$ for which
\[\mu_1\notin \textup{conv}[\sigma(A)\ominus\{\lambda_a\}].\]

\noindent\textbf{Case 4}: $\mu_1\in W(A)\setminus [ \Lambda_2(A)\cup \partial W(A)]$ and there exist $\lambda_a,\lambda_{a+1}\in \sigma(A)\setminus \Lambda_2(A)$ for which
\[\mu_1\in\textup{conv}\{\lambda_{a-1},\lambda_a,\lambda_{a+1}\}\cap\textup{conv}\{\lambda_a,\lambda_{a+1},\lambda_{a+2}\}
\] and the intersection is of two nondegenerate triangular regions. 

\textbf{Case 1} and \textbf{Case 2} are proved in Propositions \ref{inrank2} and \ref{twocases}, respectively. We can assume that the eigenvalues are distinct due to Corollary \ref{distinct}.

Let $\mu_1\in W(A)\setminus(\Lambda_2(A)\cup \partial W(A))$. If $\mu_1$ is in \textbf{Case 3}, assume $\lambda_a:=\lambda_2$. Then 
\[{\cal R}_A(\mu_1)=\textup{conv}[\{w_{123}\}\cup(\sigma(A)\ominus \{\lambda_2\})]\]
due to Proposition \ref{dec28l3} and Corollary \ref{dec29}. By Proposition \ref{realpartmup1v2}, the assertion follows. 

If $\mu_1$ is in \textbf{Case 4}, assume $\lambda_{a-1}:=\lambda_n$ and $\lambda_{a}:=\lambda_1$. Then 
\[{\cal R}_A(\mu_1)=\textup{conv}\{w_{123},w_{12n},\lambda_3,\ldots,\lambda_n\}\]
due to Proposition \ref{isogrhs}. By Proposition \ref{realpartmup2v2}, the assertion follows.
\end{proof}

\section{Concluding remarks}
In this study, we considered $2$-Ritz sets of a normal matrix having no eigenvalues in the interior of its numerical range. We identified the smallest convex region containing all $\mu_2$'s for which $\{\mu_1,\mu_2\}$ is a $2$-Ritz set. An open problem is to develop similar results for $k$-Ritz sets where $3\leq k<n-1$. 

Another open problem is to consider a normal matrix with some eigenvalues in the interior of its numerical range. In this case however, it is less clear how to characterize $\textup{conv}[{\cal B}_A(\mu_1)]$. Figure 2 shows a numerical example where $\textup{conv}[{\cal B}_A(\mu_1)]\neq \textup{conv}\{w_{123},w_{124},\lambda_3,\lambda_4\}.$
\begin{figure}[h]
\centering
\includegraphics[height=7.5cm, width=10.5cm]{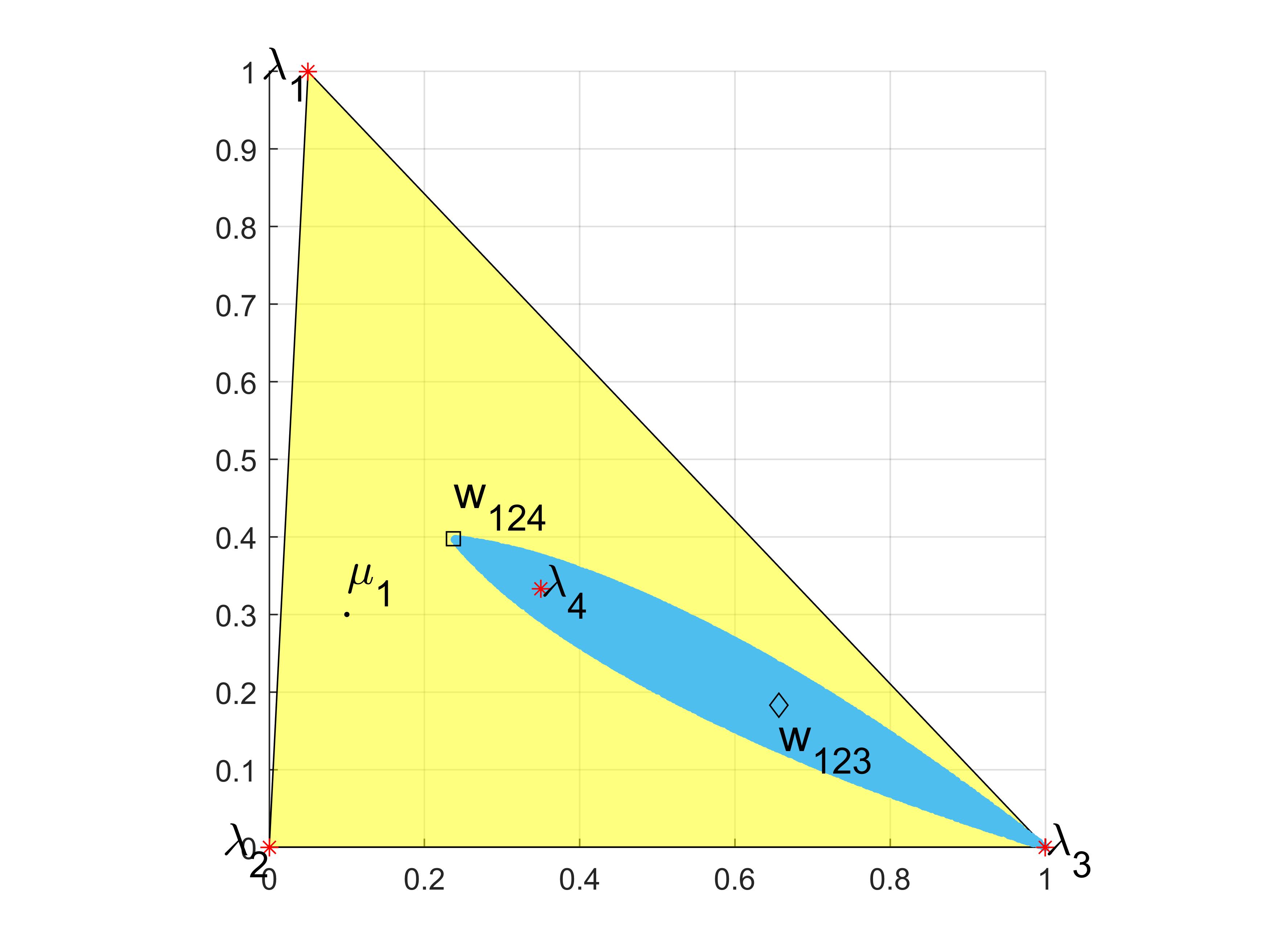}
\caption{An example where Theorem \ref{main} fails to hold when $\lambda_4\in\textup{conv}\{\lambda_1,\lambda_2,\lambda_3\}$. The blue region is the set ${\cal B}_A(\mu_1)$.}
\end{figure}

\section*{Funding}
The work of the second author was partially supported by Simons Foundation grant 355645.
\section*{Acknowledgment} We would like to thank the referee for their careful attention in reviewing this paper and for simplifying the proof of Lemma \ref{keylemma2}.

\bibliographystyle{tfnlm}
\bibliography{biblio}
\end{document}